\documentclass[12pt, linespread]{amsart}
\usepackage[T1]{fontenc}
\usepackage[utf8]{inputenc}
\usepackage[left=3cm,top=2.5cm,bottom=2.5cm,right=3cm]{geometry}
\usepackage{times}
\usepackage{amsmath}
\usepackage{amsxtra}
\usepackage{amscd}
\usepackage{amsthm}
\usepackage{amssymb}
\usepackage[normalem]{ulem}
\usepackage{comment}
\usepackage{xcolor}
\usepackage{tikz}
\usetikzlibrary{matrix,arrows,decorations.pathmorphing,automata, matrix, positioning, calc, shapes.multipart,decorations.pathreplacing,calligraphy}
\usepackage{graphicx}
\usepackage[config, labelfont={normalsize}]{caption,subfig}
\usepackage{mathrsfs}
\definecolor{green}{RGB}{0,127,0}
\definecolor{red}{RGB}{191,0,0}
\usepackage[colorlinks,cite color=red,link color=green,pagebackref=true]{hyperref}
\usepackage[shortlabels]{enumitem}
\usepackage[vcentermath]{youngtab}
\usepackage{epigraph}
\usepackage{genyoungtabtikz}
\usepackage[capitalize]{cleveref}
\usepackage{mathtools}
\usepackage{algorithm,algorithmicx,algpseudocode}
\algrenewcommand{\algorithmicrequire}{\textbf{Input:}}
\algrenewcommand{\algorithmicensure}{\textbf{Output:}}
\algnewcommand\And{\textbf{and} }
\algnewcommand\Or{\textbf{or} }

\usepackage{mynewcommands}

\numberwithin{equation}{section}

\theoremstyle{definition}
\newtheorem{theorem}[num]{Theorem}
\newtheorem{lemma}[num]{Lemma}
\newtheorem{proposition}[num]{Proposition}
\newtheorem{prop/def}[num]{Proposition-Definition}
\newtheorem{corollary}[num]{Corollary}
\newtheorem{problem}[num]{Problem}
\newtheorem{definition}[num]{Definition}

\usepackage{todonotes}

\newcommand{\calA}{\mathcal{A}}

\newcommand{\CCC}{\mathcal{C}}
\newcommand{\I}{\mathfrak{I}}

\newcommand{\ov}[1]{\overline{#1}}

\DeclareMathOperator{\col}{col}
\DeclareMathOperator{\row}{row}
\DeclareMathOperator{\shape}{shape}

\DeclareMathOperator{\Cyc}{CoCyc}
\DeclareMathOperator{\AAA}{A}
\DeclareMathOperator{\CC}{C}
\DeclareMathOperator{\Tab}{Tab}

\DeclareMathOperator{\YTab}{YTab}
\DeclareMathOperator{\STab}{STab}
\DeclareMathOperator{\SYTab}{SYTab}

\DeclareMathOperator{\SSTab}{SSTab}
\DeclareMathOperator{\SSYTab}{SSYTab}
\DeclareMathOperator{\grav}{grav}
\DeclareMathOperator{\Comp}{Comp}
\DeclareMathOperator{\Part}{Part}
\DeclareMathOperator{\shift}{shift}
\DeclareMathOperator{\wshift}{wshift}
\DeclareMathOperator{\localshift}{locshift}

\DeclareMathOperator{\ch}{ch}
\DeclareMathOperator{\red}{red}
\DeclareMathOperator{\nred}{nred}
\DeclareMathOperator{\partners}{partners}
\DeclareMathOperator{\partner}{partner}
\DeclareMathOperator{\true}{True}
\DeclareMathOperator{\false}{False}
\DeclareMathOperator{\Span}{Span}
\DeclareMathOperator{\pos}{pos}
\DeclareMathOperator{\SympTab}{SympTab}

\title[Symplectic Kostka--Foulkes polynomials I]{A positive combinatorial formula for symplectic Kostka--Foulkes polynomials I: Rows}

\author[M.~Dołęga]{Maciej Dołęga}
\address{
Institute of Mathematics, 
Polish Academy of Sciences, 
ul. Śniadeckich 8, 
00-956 Warszawa, Poland.}
\email{mdolega@impan.pl}

\author[T.~Gerber]{Thomas Gerber}
\address{\'Ecole Polytechnique F\'ed\'erale de Lausanne, 1015 Lausanne, Switzerland.}
\email{thomas.gerber@epfl.ch}

\author[J.~Torres]{Jacinta Torres}
\address{
Institute of Mathematics, 
Polish Academy of Sciences, 
ul. Śniadeckich 8, 
00-956 Warszawa, Poland.}
\email{jtorres@impan.pl}

 \thanks{MD is supported by {\it Narodowe Centrum Nauki}, grant
   UMO-2017/26/D/ST1/00186. TG is supported by the Ambizione project of the Swiss National Science Foundation. JT was supported by the { \it Deutsche Forscungsgemeinschaft} project TO 1135-1 and partially supported by  {\it Narodowe Centrum Nauki}, grant number 2017/26/A/ST1/00189 and the Max Planck Institute for Mathematics in the Sciences in Leipzig.}


\newcommand{\op}{\operatorname}

\newcount\cols
{\catcode`,=\active\catcode`|=\active
 \gdef\Young#1{\hbox{$\vcenter
 {\mathcode`,="8000\mathcode`|="8000
  \def,{\global\advance\cols by 1 &}%
  \def|{\cr
        \multispan{\the\cols}\hrulefill\cr
        &\global\cols=2 }%
  \offinterlineskip\everycr{}\tabskip=0pt
  \dimen0=\ht\strutbox \advance\dimen0 by \dp\strutbox
  \halign
   {\vrule height \ht\strutbox depth \dp\strutbox##
    &&\hbox to \dimen0{\hss$##$\hss}\vrule\cr
    \noalign{\hrule}&\global\cols=2 #1\crcr
    \multispan{\the\cols}\hrulefill\cr%
   }
 }$}}
\gdef\Skew(#1:#2){\hbox{$\vcenter
{\mathcode`,="8000\mathcode`|="8000
  \dimen0=\ht\strutbox \advance\dimen0 by \dp\strutbox
  \def\boxbeg{\vbox
    \bgroup\hrule\kern-0.4pt\hbox to\dimen0\bgroup\strut\vrule\hss$}%
  \def\boxend{$\hss\egroup\hrule\egroup}%
  \def,{\boxend\boxbeg}%
  \def|##1:{\boxend\vrule\egroup\nointerlineskip\kern-0.4pt
    \moveright##1\dimen0\hbox\bgroup\boxbeg}%
  \def\\##1\\##2:{\boxend\vrule\egroup\nointerlineskip\kern-0.4pt
    \kern ##1\dimen0\moveright##2\dimen0\hbox\bgroup\boxbeg}%
  \moveright#1\dimen0\hbox\bgroup\boxbeg#2\boxend\vrule\egroup
 }$}}
}

\begin{document}

\sloppy

\begin{abstract}
We prove a conjecture of Lecouvey, which proposes a closed, positive
combinatorial formula for symplectic Kostka--Foulkes polynomials, in
the case of rows of arbitrary weight. 
To show this, we construct a new algorithm for computing cocyclage in
terms of which the conjecture is described. Our algorithm is free of
local constraints, which were the main obstacle in 
Lecouvey's original construction.
In particular, we show that our model is governed by the situation in type A.
This approach works for arbitrary weight and we expect it to lead to a
proof of the conjecture in full generality.  
\end{abstract}

\maketitle

\section{Introduction}

The main motivation for this work is understanding an interplay
between combinatorics and representation theory which is highly
manifested in the structure of so-called \emph{Kostka--Foulkes polynomials}. 
Let $\mathfrak{g}$ be a complex, simple Lie algebra of rank $n$. Kostka--Foulkes polynomials $K_{\lambda, \mu}(q)$ are defined for two dominant integral weights as the transition 
coefficients
between two important bases of the ring of symmetric functions in the
variables $x = (x_1,..., x_n)$ over $\mathbb{Q}(q)$: Hall--Littlewood
polynomials $P_{\lambda}(x;q)$ and Weyl characters $\chi_\mu(x)$. They
are $q$-analogues of weight multiplicities \cite{Kato1982}, affine
Kazhdan--Lusztig polynomials \cite{Lusztig1983, Kato1982}, and appear
in various other situations in geometric and combinatorial
representation theory (see \cite{Brylinski1989, AnthonyLetzterZelikson2000} and references
therein). We refer the reader to \Cref{kostka} for a precise definition
of Kostka--Foulkes polynomials and recommend \cite{NelsenRam2003} as a thorough reference.  \\

Due to their interpretation as Kazhdan--Lusztig polynomials, we know
that Kostka--Foulkes polynomials have nonnegative integer coefficients. This fact leads to one of the most important unsolved problems in combinatorial representation theory:

\begin{problem}
\label{problem}
Find a set $\cT(\lambda,\mu)$ and a combinatorial 
statistic
$\ch: \cT(\lambda,\mu) \to \Z_{\geq 0}$ such that the Kostka--Foulkes
polynomial $K_{\lambda,\mu}(q)$ is the generating function of
$\cT(\lambda,\mu)$ with respect to $\ch$. In other terms find
$\cT(\lambda,\mu)$ and $\ch$ such that

\begin{align}
\label{formulaintro}
K_{\lambda,\mu}(q) = \sum_{T \in  \cT(\lambda,\mu)}q^{\ch(T)}.
\end{align}
\end{problem}

Since $K_{\lambda,\mu}(q)$ is a  $q$-deformation of weight
multiplicities then $\#\cT(\lambda,\mu) = K_{\lambda,\mu}(1)$ is the
dimension of the $\mu$-weight space of the irreducible, finite
dimensional $\mathfrak{g}$-module of highest weight $\lambda$. In particular, in order to tackle
\Cref{problem} and find an appropriate set $\cT(\lambda,\mu)$, it
seems natural to seek for an object which parametrizes the aforementioned $\mu$-weight space of the irreducible, finite
dimensional $\mathfrak{g}$-module of highest weight $\lambda$. This
approach turned out to be very succesful in type $\AAA_{n-1}$, that is
when $\mathfrak{g} = \mathfrak{sl}(n,\mathbb{C})$. 
In this case dominant integral weights are identified with partitions of at most $n-1$ parts, and a natural candidate for $\cT(\lambda,\mu)$ is the set $\op{SSYT}(\lambda,\mu)$ of semistandard Young tableaux of shape $\lambda$ and weight $\mu$. 
In this context, Foulkes conjectured the existence of such a statistic
\cite{Foulkes1974}, which was explicitly found by Lascoux and
Sch\"utzenberger \cite{LascouxSchutzenberger1978}. They called their
statistic \emph{charge} (which explains our abbreviation $\ch$ used
also in the general context of arbitrary type)
and established the celebrated formula of \Cref{problem} in type $\AAA_{n-1}$.
Let us briefly describe this statistic. We start by defining the
charge statistic $\ch$ on \textit{standard words} in the alphabet
$\mathcal{A}_{n} = \{ 1, ..., n\}$, that is words where each $i \in
\mathcal{A}_{n}$ appears exactly once. Standard words are naturally
identified with permutations by setting $w = \sigma(1) \cdots
\sigma(n)$, where $\sigma \in \frak{S}(n)$ is a permutation. We define
$\ch(w)$ --- the charge of $w$ --- recursively:
\begin{enumerate}
\item set $c(1) = 0$,
  \item for $r \geq 2$, define $c(r) = c(r-1)$ if $\sigma^{-1}(r) <
    \sigma^{-1}(r-1)$, and $c(r) = c(r-1)+1$ otherwise,
    \item set $\op{ch}(w) = \sum_{i=1}^{n}c(i)$.
    \end{enumerate}
Let $w$ be a word in the alphabet
$\mathcal{A}_{n}$ such that the number of occurrences of $i+1$ in this
word is less
or equal to the number of occurrences of $i$ for each $i+1 \in
\mathcal{A}_{n}$. For such a word, we can extract its standard
subwords $w_{1}, ..., w_{m}$ as follows: the first subword $w_1$ of
$w$ is obtained by selecting the rightmost $1$ in $w$, then the
rightmost $2$ appearing to the left of the selected $1$, and so on
until there is no $k + 1$ to the left of the current value $k$ being
selected. At this point, we select the rightmost $k + 1$ in $w$ and
continue with the previous process until the largest value appearing
in $w$ is reached. The subword $w_1$ is obtained by erasing all the
letters from $w$ that were not selected and we proceed by selecting
$w_2$ by the same procedure performed on the word consisting of the
letters that were not selected so far. We continue, until no letters
are left. Finally, we will define $\ch(w)$ by setting $\op{ch}(w)  = \sum_{i=1}^{m}\op{ch}(w_i)$. One can show that $\op{ch}$ is constant on Knuth equivalent words (see
\cite[Proposition 2.4.21]{Butler1994}), which allows to define $\ch$
as a statistic on semistandard Young tableaux with partition weight. In practice, if $T \in
\op{SSYT}(\lambda,\mu)$ is a semistandard Young tableau of shape
$\lambda$ and weight $\mu$, one may define $\op{ch}(T)$ as
$\op{ch}(\op{w}(T))$, where $\op{w}(T)$ is its south-western row word \footnote{We warn the
  reader that we will work solely with north-eastern column words in
  the remaining sections of this text. However, to be consistent with
  the definition of the charge statistic on words
  \cite{LascouxSchutzenberger1978, Butler1994}, and to avoid reading
  words backwards, we prefer to stick to south-western row words to define charge.}
  
\begin{exa}
Let 
$T=
\Yboxdim{1cm}
\begin{tikzpicture}[font=\tiny, scale=0.3,baseline=0cm]
\tgyoung(0cm,1cm,<1><1><1><2><3>,<2><2><4>,<3><5>)
\end{tikzpicture}.$
The south-western row word $\op{w}(T)$ of $T$ is $3522411123$. From it we may
extract the subwords $\op{w}_{1} = 35241$, obtained as
$\widehat{3}\widehat{5}2\widehat{2}\widehat{4}11\widehat{1}23$ (we mark selected letters by
a hat), $\op{w}_{2} = 213$, obtained as
$\widehat{2}1\widehat{1}2\widehat{3}$, which finally gives $\op{w}_{3} = 12$. 
Their charges are $\op{ch}(\op{w}_{1}) = 2, \op{ch}(\op{w}_{2}) = 1$ and $\op{ch}(\op{w}_3) = 1$, respectively. 
Therefore $\op{ch}(T) = \op{ch}(3522411123) = 2+1+1 = 4$. 
\end{exa}

A thorough introduction to Kostka--Foulkes polynomials in type
$\AAA_{n-1}$ and the charge statistic from a purely combinatorial
point of view is carried out in \cite{Macdonald1995}. We  refer the
reader to \cite{Butler1994} for a beautiful exposition and proof of
(\ref{formulaintro}), which makes use of a recursive formula for
computing Kostka--Foulkes polynomials due to Morris
\cite{Morris1963}. The aforementioned recursive formula, in turn, is
deduced from a formula for Hall--Littlewood polynomials discovered by
Littlewood \cite{Littlewood1961}. It is worth mentioning here that there
are various generalizations of \cref{problem} in type $\AAA$ leading
to many open problems, see
\cite{Macdonald1988,LascouxLeclercThibon1997,Haiman2001,LapointeLascouxMorse2003,GrojnowskiHaiman2007,Dolega2019}
among others.
 \\

In this work, we focus on \Cref{problem} for type $\CC_{n}$, that is, in case of the symplectic Lie algebra $\mathfrak{g} = \mathfrak{sp}(2n,\mathbb{C})$. 
To the best of our knowledge this is the only case of \Cref{problem}
having an explicit conjectural solution, which was formulated by
Lecouvey in \cite{Lecouvey2005}. In this case, the dominant integral
weights $\lambda,\mu$ can again be identified with partitions of at
most $n$ parts, however there are several natural combinatorial
candidates for the set $\cT(\lambda,\mu)$ such as King tableaux
\cite{King1976}, De Concini tableaux \cite{DeConcini1979} or
Kashiwara--Nakashima tableaux \cite{KashiwaraNakashima1994} that we
also call \emph{symplectic tableaux}. The last model denoted
$\op{SympTab}_{n}(\lambda,\mu)$ will be of
particular importance in this paper as Lecouvey's conjecture is
formulated in terms of symplectic tableaux. These are defined to be
semistandard Young tableaux with some additional constraints (see \Cref{symptab} and
\cite{Lecouvey2005}) and entries in the ordered
alphabet \[\mathcal{C}_{n} = \{\overline{n} < \dots < \overline{1} <
  1 < \dots < n \}, \]
such that the shape of a tableau is given by $\lambda$ and its weight
by $\mu$. Here, the weight of a tableau with
entries in $\cC_n$ is defined slightly differently than the weight of a
tableau of type $\AAA_{n-1}$ and is given by the vector $(a_{\overline{n}}, ..., a_{\overline{1}})$, where $a_{\overline{i}}$
is the difference between the number of occurrences of
$\overline{i}$'s and $i$'s in $T$. Lecouvey defined a charge statistic
$\ch_{n}: \op{SympTab}_{n}(\lambda,\mu) \rightarrow \mathbb{Z}_{\geq
  n}$ by analogy to the situation in type $\AAA_{n-1}$. Before we
describe Lecouvey's construction, which might seem quite technical, let us first
recall this specific situation in type $\AAA_{n-1}$ to motivate the
reader\footnote{The cocyclage described here is the one used in
  Lecouvey's paper~\cite{Lecouvey2005}. Note that there are some other
  variants, see for instance~\cite{LascouxSchutzenberger1978,Butler1994,Lothaire2002}.}.
The idea is that there is a procedure, known as \textit{cocyclage}
and denoted $\op{CoCyc}_{\AAA}$, 
whose successive applications to any semistandard Young tableau
yield a tableau whose weight is equal to its shape,
see \Cref{sec:CyclageA} for details.
This defines a poset structure on the set of semistandard Young tableaux:
$T \to T'$ whenever $T' = \op{CoCyc}_{\AAA}(T)$. It is readily shown that whenever $T \to T'$ then
$\op{ch}(T') = \op{ch}(T) -1$. 
Moreover, one easily checks that if the shape and the weight of $T$ coincide, 
then $\op{ch}(T)=0$, therefore 
$\op{ch}(T)=k$ where $k\geq0$ is the smallest integer such that 
$\op{CoCyc}^{k}_{\AAA}(T)$ has weight equal to its shape.
This way, we can compute charge without referring to the
standard subwords. In \Cref{typeacyc} we show that the cocyclage 
poset in type $\AAA$ carries an additional structure in terms of unimodal compositions. 
This structure is ``lifted'' to type $\op{C}$, as is outlined in \Cref{mainintro} below. 
\\

Before we describe Lecouvey's conjectural solution to \Cref{problem}
involving cocyclage it is worth mentioning already proposed partial
solutions to \Cref{problem} in type $\CC_n$. In \cite{Lecouvey2006}
Lecouvey defined a certain statistic on the set of
Kashiwara--Nakashima tableaux $\op{SympTab}_{n}(\lambda,\mu)$ and he
showed that for some special pairs $(\lambda,\mu)$ his statistic
indeed gives the
correct answer for \Cref{problem}. However, he also showed that in
general, his statistic does not give the correct answer for this
problem. On the other hand the solution in the weight zero case has been given recently in \cite[Theorem 6.13]{LecouveyLenart2018},
using the aforementioned combinatorial model of $\cT(\lambda,\mu)$ called King tableaux. Their result relies on an interpretation of Kostka--Foulkes
polynomials in terms of generalized exponents which holds in this special case of weight zero. Notably, their formula relies in a deciding way on a formula for generalized exponents in terms of branching multiplicities, and their methods suggest explicit ways in which several branching rules \cite{Sundaram1990, Kwon2018, SchumannTorres2018} could be related to each other.

\subsection{Main result and methodology} \label{mainintro} In order to define the statistic \[\op{ch}_{n}: \op{SympTab}_{n}(\lambda,\mu) \rightarrow \mathbb{N}\] and formulate his conjecture, 
Lecouvey proceeded by analogy to the situation in type $\AAA_{n-1}$ described above. He used a symplectic version of column insertion, which he introduced in \cite{Lecouvey2002}, to define a symplectic cocyclage operation $\op{CoCyc}_{\CC}$ which transforms a symplectic tableau $T \in \op{SympTab}_{n}$ into a symplectic tableau $\Cyc_{\CC}(T) \in \op{SympTab}_{m}$ for 
$m\geq n$.
The statistic $\op{ch}_{n}$ is defined as follows. Let $T \in \op{SympTab}_{n}$ be a symplectic tableau. In \cite{Lecouvey2005}, Lecouvey showed that there exists a non-negative integer $m$ such that $\Cyc_{\CC}^{m}(T)$ is a column $C(T)$ of weight zero. We denote by $m(T)$ the smallest non-negative integer with this property. For a symplectic column $C$ of weight zero we set $E_{C}= \{ i \geq 1| i \in C, i+1 \notin C\}.$ The charge of $C$ is defined by

\[ \op{ch}_{n}(C) = 2 \underset{i \in E_{C}}{\sum} (n-i),\]
and the charge of an arbitrary symplectic tableau $T$ is defined by 

\[\op{ch}_{n}(T) = m(T) + \op{ch}_{n}(C(T)).  \]Lecouvey
\cite{Lecouvey2005} conjectured the following solution of \Cref{problem} in type $\CC_{n}$:
\begin{conj}
\label{conji}
Let $\mu,\lambda$ be partitions with at most $n$ parts. Then 
\begin{align}
\label{conj}
 K^{\CC_n}_{\lambda,\mu}(q) = \underset{T \in \op{SympTab}_{n}(\lambda,\mu)}\sum q^{\op{ch}_{n}(T)}.    
\end{align}
\end{conj}

Our main result reads as follows.

\begin{theorem}
\label{mainresultintro}
Let $\lambda = (p)$ and $\mu = (\mu_{\ov n},\dots,\mu_{\ov 1})$ be an
arbitrary partition. Then \Cref{conji} holds true:
\[K^{\CC_{n}}_{\lambda,\mu}(q) = \sum_{T \in
    \op{SympTab}_{n}(\lambda,\mu)} q^{\op{ch}_{n}(T)}. \]
\end{theorem}

A pivotal point in our methodology, and one which we expect will have impact on the study of the general case of \Cref{conji}, 
is a reformulation of Lecouvey's construction in the setting of
\Cref{mainresultintro} by
providing a new algorithm to compute $\Cyc^k_{\CC}(T)$ for
arbitrary integer $k$. The
big advantage of our approach is that in \Cref{alg:content}, which
completes this task, we are able to eliminate
local constraints which appear in the original construction in two
different contexts:
\begin{itemize}
  \item[$\bullet$] we need to compute $\Cyc^{k-1}_{\CC}(T)$ in order to compute
    $\Cyc^{k}_{\CC}(T)$;
    \item[$\bullet$] for each column of $\Cyc^{k-1}_{\CC}(T)$, we need to insert boxes recursively into
      consecutive subcolumns of size $2$.
    \end{itemize}
In order to free ourselves from the second constraint we give a formula for
inserting an entry into a whole column at once, which is given by \Cref{prop:Insertion}. 
Although more technical in appearance, our new definition allows us to
have a full grasp of the symplectic cocyclage procedure. We show in
\Cref{thm:crucial} that for a partition $\lambda = (p)$ which consists of one row
and for an arbitrary partition $\mu$ the symplectic tableau $\Cyc^{k}_{\CC}(T)$, where $T \in
\op{SympTab}_{n}(\lambda,\mu)$, is given by \Cref{alg:content}. The
main philosophy of \Cref{alg:content} is that in order to compute
$\Cyc^{k}_{\CC}(T)$, it is enough to only apply $\Cyc_{\AAA}$ to certain
standard Young tableaux and then apply a very simple function
which changes the entries of the outcome. \\

As an application, we are able to
compute $\op{ch}_{n}(T)$ directly from $T$ and, using
simple recurrence for Hall--Littlewood polynomials of type $\CC_{n}$ proved
by Lecouvey in \cite[Theorem 3.2.1.]{Lecouvey2005}, we deduce
\Cref{mainresultintro}. We believe that our strategy might lead
us to the solution of \Cref{conji} in full generality. Indeed, the restriction $\lambda = (p)$ is due to the
fact that symplectic tableaux of row shape coincide with semistandard
tableaux with entries in the alphabet $\mathcal{C}_{n}$ (see
\Cref{prop/def}). In particular, there exists a unique standard tableau of
shape $(p)$, and \Cref{alg:content} consists in applying $\Cyc_{\AAA}$ multiple
times to this unique tableau. It seems likely that in the more general
case there exists a ``right''
labelling of the boxes of any symplectic tableau $T$ of arbitrary shape,
such that a very similar procedure could be followed to compute $\Cyc_{\CC}^k(T)$
and therefore $\op{ch}_{n}$. So far, this question remains open and we will be investigating this in the future.
\\

\subsection{Organization of the paper. } In \Cref{preliminaries} we
introduce all the necessary combinatorial preliminaries to follow the rest of the paper. Moreover, in \Cref{typeacyc}, we show that the cocyclage 
in type $\AAA_{n-1}$ carries an extra structure in terms of unimodal compositions. In \Cref{typec}, we introduce necessary
combinatorics in type $\CC_{n}$, including the original definition of
insertion and its non-recursive form given by \Cref{prop:Insertion}. 
We also present the cocyclage algortithm of Lecouvey, the definition of the charge statistic on symplectic tableaux, 
and state his conjectural positive formula for symplectic Kostka--Foulkes polynomials. 
In \Cref{insandshift} we describe \Cref{alg:content} producing a
certain tableau which we show coincides with the tableau obtained from
a row tableau by performing the cocyclage operation $k$ times. We
conclude this section by proving Lecouvey's conjecture for $\lambda = (p)$
and arbitrary $\mu$ in \Cref{onerow}, which follows from our
algorithmic description. Since the proof of
our description of the cocyclage in type $\CC$ is the most technically
involved part of the paper, we present it in
a separate \cref{appendix}.

\section{Preliminaries}
\label{preliminaries}

\subsection{Tableaux}
\label{tableaux}
\emph{A composition} $\alpha \vDash l$ of \textit{size} $l \in \Z_{\geq 0}$ is a
sequence of non-negative integers $\alpha = (\alpha_1,\alpha_2,\dots)
\in \Z_{\geq 0}^{\Z_{>0}}$ such that $\sum_{i}\alpha_i = l$ and such
that $\alpha_i=0$ implies that
$\alpha_{i+1}=0$ for any $i \in \Z_{>0}$. 
In particular, there are only finitely many non-zero $\alpha_i$ and we denote their number by
$\ell(\alpha)$ calling it the \emph{length} of the composition $\alpha$. 
We will also use the notation $|\alpha| = l$. We denote the set of compositions of size $n$ by
$\Comp_l$, and we set $\Comp = \bigcup_{l}\Comp_l$. For any positive integer $i \in \Z_{>0}$ and for any
composition $\alpha \in \Comp_l$ we define a new composition
$\alpha-i$ as follows:
\[ \alpha-i = \begin{cases} \alpha &\text{ if $\alpha_j \neq i$ for
      all $j \in \Z_{>0}$};\\
 (\alpha_1,\dots,\alpha_{j-1},\alpha_{j+1},\dots) &\text{ where $j =\min\{k: \alpha_k = i\}$}.\end{cases}\]
For convenience we denote the unique composition $(0,0,\dots)$ of size $0$ by $0$. 
To any $\alpha \in \Comp_l$ we associate its diagram
defined by:
\[ \cD_\alpha = \{(i, j):1 \leq i \leq \alpha_{-j}, j \leq -1 \} \subset \mathbb{Z}_{>0} \times \mathbb{Z}_{<0}.\]
The elements of $\cD_\alpha$, referred to  and denoted as boxes, 
are linearly ordered by the so-called \emph{natural order}, 
starting from the north-western box and
reading boxes row by row from left to right. Formally, this is the
following variant of the lexicographic order: $(i_1,j_1) \leq (i_2,j_2)
\iff j_1 > j_2$ or $j_1 = j_2, i_1 \leq i_2$. 
For $c \in [1,|\al|]$ we denote the $c$-th box of
$\cD_\alpha$ in natural order by $\square_c$, or by $c$ when it does
not lead to confusion.

\begin{exa}
Let $\alpha = (3,1,2) \in \Comp_{6}$.  The set of boxes defined by $\mathcal{D}_{\alpha}$ is pictured below, where each $(i,j) \in \mathcal{D}_{\alpha}$ lies in its corresponding box. 
\begin{center}
\begin{tikzpicture}
  \draw[help lines] (0,0) grid (3,-3);
  \draw[thick] (0,0) -- (0,-3) -- (2,-3)-- (2,-2)--(1,-2)--(1,-1)--(3,-1)--(3,0)--(0,0);
\node[] at (0.5, -0.5) {\tiny{(1,-1)}};
\node[] at (1.5, -0.5) {\tiny{(2,-1)}};
\node[] at (2.5, -0.5) {\tiny{(3,-1)}};
\node[] at (0.5, -1.5) {\tiny{(1,-2)}};
\node[] at (0.5, -2.5) {\tiny{(1,-3)}};
\node[] at (1.5, -2.5) {\tiny{(2,-3)}};
\end{tikzpicture}
\end{center}

The elements of $\mathcal{D}_{\alpha}$ are ordered with respect to the natural order as:
\[ (1,-1) < (2,-1) < (3,-1) < (1,-2) < (1,-3)< (2,-3).\]
We will sometimes identify $\alpha$ and $\cD_\al$.
\end{exa} 

Let $(\calA,\prec)$ be a linearly ordered alphabet with minimal element $a$.
For any composition $\alpha \vDash l$ we define a \emph{tableau} $T$ of
shape $\alpha$ and entries in $\calA$ to be a filling of the boxes of
the diagram of $\alpha$ by elements from alphabet $\calA$. Formally, $T$ is a function 
\[T: \cD_\alpha \rightarrow \mathcal{A}.\] 
The \emph{content} of a tableau $T$ of shape $\alpha$ is the multiset of its entries.
When
$\calA$ is a countable ascending chain (with minimal element $a$), we say that a tableau has \emph{weight} $\beta
= (\beta_1,\beta_2,\dots)$ when its content is given by the multiset
\[
  \{a^{\beta_1},(a+_\succ)^{\beta_2},\dots,
  (a+_{\succ^k})^{\beta_{k+1}},\dots\},\] 
where $a+_{\succ}$ denotes the successor of $a$, and $a+_{\succ^{k+1}}= a+_{\succ^k}+_{\succ}$.
We
call a tableau \emph{semistandard} if for any pair of boxes lying in the
same row the content of the left box is smaller than or equal to the
content of the right box, and such that for any pair of boxes lying in
the same column the content of the upper box is smaller than the
content of the lower box, that is such that $T(i,j) \leq T(i+1,j)$ and
$T(i,j) < T(i,j-1)$. We call a tableau \emph{standard} if it is
semistandard of weight $\beta$ which is a \emph{column}, that is, $\beta = (1,1,\dots,1)$.

\begin{definition}
\label{def:NaturalTableau}
We call a tableau \emph{natural} if it is
semistandard and if it has the property that for boxes $a \leq b$ in the
natural order $T(a) \leq T(b)$.
\end{definition}

\begin{exa}
Let $\mathcal{A} = \left\{a, b,c \right\}$ with the linear order given by $a \prec b \prec c$, let $\alpha = (3,1,2)$ and $T = $
\Yboxdim{1cm}
\begin{tikzpicture}[font=\tiny, scale=0.3,baseline=0cm]
 \tgyoung(0cm,0.5cm,<a><a><b>,<b>,<c><c>)
\end{tikzpicture}, that is, $T((1,-1)) = a = T((2,-1)), T((3,-1))  = b
= T((1,-2)),$ and $T((1,-3)) = c =T((2,-3))$. Then $T$ is a
semistandard, natural tableau of shape $\alpha$ with entries in $\mathcal{A}$ and weight $(2,2,2)$.  
\end{exa}

We are particularly interested in compositions with some additional
properties. We call a composition $\alpha$ \emph{unimodal} if it is
unimodal as a sequence, that is there exists $j \in \Z_{>0}$ such that
$\alpha_1 \leq \cdots \leq \alpha_j \geq \alpha_{j+1} \geq \cdots$. A
\emph{partition} is a composition with non-increasing elements (in particular, partitions are unimodal). 
Its diagram is called a \emph{Young diagram}. A partition $\la$ of size $l$ is denoted by $\la \vdash l$. 
We denote the set of partitions of size $l$ by $\Part_l$ and $\Part =
\bigcup_l\Part_l$. 
Finally we denote the set of tableaux (semistandard and standard tableaux, respectively) of shape $\alpha$ and weight $\beta$ by
$\Tab_{\mathcal{A}}(\alpha,\beta)$ ($\SSTab_{\mathcal{A}}(\alpha,\beta), \STab_{\mathcal{A}}(\alpha)$,
respectively) and we denote by $\Tab_l(\mathcal{A}), \SSTab_l(\mathcal{A}), \STab_l(\mathcal{A}),$ ($\YTab_l(\mathcal{A}),
\SSYTab_l(\mathcal{A}), \SYTab_l(\mathcal{A}), $, respectively) the set of tableaux, semistandard
tableaux, standard tableaux (Young tableaux, semistandard
Young tableaux, standard Young tableaux, respectively) of size $l$, that is
\begin{align*}
\Tab_l(\mathcal{A}) &= \bigcup_{\alpha,\beta \vDash l}\Tab_{\mathcal{A}}(\alpha,\beta), \qquad \YTab_l(\mathcal{A}) = \bigcup_{\lambda \vdash l, \beta \vDash
  l}\Tab_{\mathcal{A}}(\lambda,\beta),\\
  \SSTab_l(\mathcal{A}) &= \bigcup_{\alpha,\beta \vDash l}\SSTab_{\mathcal{A}}(\alpha,\beta),
  \qquad  \SSYTab_l(\mathcal{A}) = \bigcup_{\lambda \vdash l,
  \beta \vDash l}\SSYTab_{\mathcal{A}}(\lambda,\beta),
  \\ \STab_l(\mathcal{A}) &= \bigcup_{\alpha\vDash l}\STab_{\mathcal{A}}(\alpha), \qquad
               \SYTab_l(\mathcal{A}) =
  \bigcup_{\lambda\vdash l}\SYTab_{\mathcal{A}}(\lambda).
\end{align*}

We will drop the index $l$ to denote the corresponding union over all positive integers $l$.
Moreover, when the alphabet is clear from the context, we will drop $\cA$ in these notations.

\subsection{Augmented tableaux}
\label{augmenteddefs}

An \textit{augmented composition} is 
the data of a composition $\al$ and a box $b=(i,j)$ in the diagram of $\al$, called the \textit{augmented box}.
In this case, the augmented composition $(\al,b)$ is also called an \textit{augmentation}
of  $\al$.
The diagram of $(\al,b)$ is defined as
$$\cD_{(\al,b)}=\cD_\al\setminus \{b\}\sqcup \{b_-,b_+\}$$
where $b_-=(i-1/2,j)$ and $b_+=(i+1/2,j)$, 
and is represented by the diagram of $\al$ in which box $b$ is split into two boxes $b_-$ and $b_+$.
In particular, $(\al,b)$ has $|\al|+1$ boxes, which are again totally ordered by the natural order,
and we have $b_-=\square_c$ and $b_+=\square_{c+1}$ for some label $c\in[1,|\al|+1]$.
We will call $b_-$ and $b_+$ the augmented boxes of $\al$.
\begin{exa}
The augmented composition $((1,3),(2,-2))$ has diagram 
\[\cD_\al=\left\{(1,-1),(1,-2),(3/2,-2),(5/2,-2),(3,-2)\right\},\] which is represented by
\Yboxdim{1cm}
\begin{tikzpicture}[font=\tiny, scale=0.3,baseline=0cm]
\tgyoung(0cm,1cm,<>,<><><>)
\draw (1,0) -- (2,1);
\end{tikzpicture}.
\end{exa}
In turn, an \textit{augmented tableau} $T$ is the filling of a diagram of an augmented composition by elements of $\cA$.
Formally, $T$ is a function $\cD_{(\al,b)}\ra\cA$.
An augmented tableau $T$ of shape $(\al,b)$ induces two regular tableaux $T_-$ and $T_+$ of shape $\al$ defined by
$$
\begin{array}{cc}
\begin{array}{rrcl}
T_-:&\cD_\al & \ra & \cA \\
&c&\mapsto&
\left\{
\begin{array}{ll}
T(c) & \text{ if } c\neq b \\
T(b_-) & \text{ if } c= b
\end{array}
\right.
\end{array}
&
\begin{array}{rrcl}
T_+:&\cD_\al & \ra & \cA \\
&c&\mapsto&
\left\{
\begin{array}{ll}
T(c) & \text{ if } c\neq b \\
T(b_+) & \text{ if } c= b
\end{array}
\right.
\end{array}
\end{array}
$$
\begin{rem}\label{aug_tab}
The augmented tableau $T$ is determined by the tableau $T_+$, the box $b$ and the entry $j\in\cA$ such that $T(b_-)=T_-(b)=j$. 
\end{rem}
We represent the augmented tableau $T$ 
by the tableau $T_-$ (or equivalently $T_+$) in which we replace box $b$ by the split box 
\Yboxdim{1cm}
\begin{tikzpicture}[font=\tiny, scale=0.4,baseline=0cm]
\tgyoung(0cm,0cm,<>)
\draw (0,0) -- (1,1);
\draw (0.3,0.7) node (a) {$j$};
\draw (0.7,0.3) node (b) {$i$};
\end{tikzpicture}, where $j$ is as in \Cref{aug_tab}.\\

For any (augmented) tableau $T$, we will denote its shape by $\shape(T) \in \op{Comp} \cup \op{Comp}^{+}$. 
For a composition $\al\vDash l$, we denote $\Tab^+_\cA(\al)$ the set of augmented tableaux of shape $\al^+$ 
for some augmentation $\al^+$ of $\al$ and weight $\be\vDash l+1 $, and we call $l$
the \textit{size} of $T\in\Tab^+(\al)$.
As before, we will denote the set of all augmented tableaux of size $l$ by 
\[\Tab^{+}_{l}(\cA) = \bigcup_{\alpha \vDash l} \Tab^+_\cA(\al) .\]

\subsection{Gravity}
\label{section:gravity}

Reordering the parts of a composition $\al\vDash l$ gives a partition $\la\vdash l$.
Note that $\la$ can be also seen as the result of lifting all the
boxes in each column of $\al$ so that after the lift, the boxes in the given column are lying in consecutive rows starting from the first row.
For this reason, we denote by $\grav$ the map $\Comp_l\to\Part_l, \al\mapsto\la$
and call it the gravity map.
This description induces a map $\Tab_l \to \YTab_l$ on tableaux, which
restricts to a map $\SSTab(\alpha,\beta) \to
\SSYTab(\grav(\alpha),\beta)$ and which we denote by the same symbol.
\begin{exa}
\Yboxdim{1cm}
We have
$
\grav \left(
\begin{tikzpicture}[font=\tiny, scale=0.3,baseline=0cm]
\tgyoung(0cm,1cm,<1>,<2><3>,<4><4><5><6>,<5>)
\end{tikzpicture}\right)
=
\begin{tikzpicture}[font=\tiny, scale=0.3,baseline=0cm]
\tgyoung(0cm,1cm,<1><3><5><6>,<2><4>,<4>,<5>)
\end{tikzpicture}
$.
\end{exa}

\subsection{Shifting}
\label{section:shifting}

Let $l\in\Z_{\geq0}$ and define $\shift: \Comp_l \to \Comp_l$
as follows 
\[ \shift(\alpha) = \begin{cases} \alpha &\text{ if $\alpha = (1^l,0,\dots)$ for
      some $l \in \Z_{\geq 0}$};\\
\alpha - e_i+e_{i+1} &\text{ otherwise};\end{cases}\]
where $e_i = (\underbrace{0,\dots,0}_{\text{$i-1$ times}},1,0\dots)$
and $i = \min\{j \mid \alpha_j = \max_k \alpha_k\}$.
Geometrically, it can be interpreted as removing the rightmost upper
box from a diagram $\alpha$ and adding a box at the end of the next
row. Note that $\shift$ clearly preserves the subset of
unimodal compositions.

\begin{exa} Let 
$\alpha = 
\begin{tikzpicture}[font=\tiny, scale=0.3,baseline=0cm]
\tgyoung(0cm,1cm,<>,<><><>,<>)
\end{tikzpicture}$
. We have
$\shift(\alpha)=
\begin{tikzpicture}[font=\tiny, scale=0.3,baseline=0cm]
\tgyoung(0cm,1cm,<>,<><>,<><>)
\end{tikzpicture}.$
\end{exa}

The shift operator induces a map on natural tableaux (see \cref{def:NaturalTableau}): 
given a natural tableau $T$ of shape $\alpha$, $\op{shift}(T)$ is the unique natural tableau of shape $\shift(\alpha)$ and same entries as $T$. 

\begin{exa} Take $\cA=\{1,2,3,4\}$ and let 
$T = 
\begin{tikzpicture}[font=\tiny, scale=0.3,baseline=0cm]
\tgyoung(0cm,1cm,<1>,<2><2><3>,<4>)
\end{tikzpicture}.$
We have
$\shift(T)=
\begin{tikzpicture}[font=\tiny, scale=0.3,baseline=0cm]
\tgyoung(0cm,1cm,<1>,<2><2>,<3><4>)
\end{tikzpicture}.$
\end{exa}

Given a composition $\alpha$ and a partition $\mu$, we are interested
in the following algorithm, which will produce a new composition. We
will apply the shift operator to $\alpha$, unless the maximum size of
its parts is equal to  $\mu_{1}$. If this comes to be the case, we
remove the first  part  of size $\mu_{1}$ from $\alpha$,  and we update
$\mu$ by removing $\mu_{1}$
from it. We repeat this procedure until the largest part of $\alpha$
and $\mu$ are different. This step of the procedure is formally
described by \cref{alg:red}. In \Cref{lem:ShiftTerminates} we show that this algorithm in fact terminates.  We think of our algorithm as repeated application of a \textit{weighted} shift operation. 

\begin{exa}
\label{weightedshift}
Let $\alpha = (3,2,1)$ and $\mu = (2,2)$. We first apply $\shift$ two times: 

$
\Yboxdim{1cm}
\shift^{2}
 \begin{tikzpicture}[font=\tiny, scale=0.3, baseline = 0.35cm]
\tgyoung(0cm,1cm,<><><>,<><>,<>)
\end{tikzpicture}  = \shift
 \begin{tikzpicture}[font=\tiny, scale=0.3, baseline = 0.35cm]
\tgyoung(0cm,1cm,<><>,<><><>,<>)
\end{tikzpicture}  = 
 \begin{tikzpicture}[font=\tiny, scale=0.3, baseline = 0.35cm]
\tgyoung(0cm,1cm,<><>,<><>,<><>) = 
\end{tikzpicture}
$,
which is the minimal number of shifts of $\alpha$ to obtain a
composition whose maximum part is equal to $2 = \mu_{1}$. At
this step we remove the first part of $\shift^2(\alpha)$, which is the first part of size $2$, to obtain 
\Yboxdim{1cm}
 \begin{tikzpicture}[font=\tiny, scale=0.3, baseline = 0.35cm]
\tgyoung(0cm,1cm,<><>,<><>) = 
\end{tikzpicture}, and we update $\mu = (2)$. Since the largest
parts of $\alpha$ and $\mu$ are still equal, we remove them again to
obtain $\alpha = (2)$ and $\mu = \emptyset$. This part of the algorithm corresponds to
$\op{simp}(\shift^2(\alpha),\mu)$ given by \cref{alg:red}. We can now
apply $\shift$ to $\alpha$ to obtain 
$
\Yboxdim{1cm}
\shift
 \begin{tikzpicture}[font=\tiny, scale=0.3, baseline = 0.35cm]
\tgyoung(0cm,1cm,<><>)
\end{tikzpicture}
=
 \begin{tikzpicture}[font=\tiny, scale=0.3, baseline = 0.35cm]
\tgyoung(0cm,1cm,<>,<>)
\end{tikzpicture}
$, which finishes our algorithm since columns are by definition fixed
points for $\shift$. Therefore our algorithm terminates after $3$
applications of the weighted shift operator.
\end{exa}

We now give a formal definition of our algorithm. We first define the operator

\begin{align*}
\op{simp}:\Comp\times\Part
\to \Comp\times\Part
\end{align*}
recursively as follows. 

\begin{algorithm}[H]
\caption{Defining $\op{simp} (\alpha,\mu)$.}
\label{alg:red}
\begin{algorithmic} 
\Require A partition $\mu$ and a composition $\alpha$.
\Ensure A pair $(\beta,\nu) \in \Comp\times\Part$.\\
$\beta = \alpha$\\
$\nu = \mu$
\While{$\max \beta_{k} = \nu_{1}$}
  \State $\nu = \nu \setminus \nu_{1}$
  \State $\beta = \beta \setminus
   \max \beta_{k}$
\EndWhile
\end{algorithmic}
\end{algorithm}

Note that $\op{simp}$ corresponds to a succesive removal of the largest
parts in $\alpha$ and $\mu$ until they are different. 
Now, each step of our weighted shift algorithm may be described by the operator:

\[ \wshift(\alpha,\mu) = \begin{cases} (\shift(\alpha),\mu) &\text{ if
      $(\alpha,\mu) =  \op{simp} (\alpha,\mu)$};\\
 \left(\shift\big(\op{simp} (\alpha,\mu)_1\big), \op{simp}(\alpha,\mu)_2\right) &\text{ otherwise};\end{cases}\]
where $\op{simp} (\alpha,\mu)_i$ denotes the $i$-th coordinate of $\op{simp} (\alpha,\mu)$.

\begin{rem}
Note that  $\wshift(\alpha,0) = (\shift(\alpha),0)$. 
\end{rem}

As in the case of $\shift$, the map $\wshift$ induces a map
on the set of tableaux whose weight is a partition, which we denote by the same symbol.
More precisely, if $\al$ is the shape of $T$ and $\mu$ its weight,
the shape of $\wshift(T)$ is $\wshift(\al,\mu)_1$
and the weight of $\wshift(T)$ is $\wshift(\al,\mu)_2$.

\begin{exa}\label{ex:wazne} Take $\cA=\{1,2,3,4\}$ and
$T= 
\begin{tikzpicture}[font=\tiny, scale=0.3,baseline=0cm]
\tgyoung(0cm,1cm,<1><1><1>,<2><2><3>,<4>)
\end{tikzpicture},$
so that $\al=(3,3,1)$ and $\mu=(3,2,1,1)$.
Then
$\wshift(T)=
\begin{tikzpicture}[font=\tiny, scale=0.3,baseline=0cm]
\tgyoung(0cm,1cm,<1><1>,<2><3>)
\end{tikzpicture}.$
\end{exa}

\begin{lemma}
\label{lem:ShiftTerminates}
For any pair $(\alpha,\mu) \in \Comp\times\Part$ there exists an
integer $m$ and a partition $\nu$ such that $\wshift^{m}(\alpha, \mu) = ((1^l), \nu)$ and is a fixed
point of $\wshift$ (for some $l
\geq 0$), that is $\nu_1 \neq 1$.
\end{lemma}

\begin{proof}
We define some variation of the lexicographic order $\geq_{lex}$ on $\Comp\times\Part$ as follows:
$(\alpha,\mu) > (\beta,\nu)$ if and only if $\mu \geq_{lex} \nu$ and
$\max_k\alpha_k > \max_k\beta_k$ or $\max_k\alpha_k = \max_k\beta_k = s$
and $\min\{j:\alpha_j = s\} < \min\{j:\beta_j = s\}$. Now,
notice that 
\begin{itemize}
\item for any pair $(\alpha,\mu) \in \Comp\times\Part$, we have
  $(\alpha,\mu)>\wshift(\alpha,\mu)$ or $\wshift(\alpha,\mu) =
  (\alpha,\mu)$;
\item
for any pair $(\alpha,\mu) \in \Comp\times\Part$, we have $|\wshift(\alpha,\mu)| \leq
|(\alpha,\mu)|$, where $|(\alpha,\mu)| = |\alpha|+|\mu|$. 
\end{itemize}
In particular the set $\{\wshift^k (\alpha,\mu): k \in \Z_{\geq 0}\}$ is finite, 
and there exists $k \in \Z_{\geq 0}$ such that $\wshift^{k+1}(\alpha,\mu) =
\wshift^{k}(\alpha,\mu)$. But the only fixpoints of $\wshift$ are of the form
  $((1^l),\nu)$ for some $l \leq |\alpha|$ and $\nu_1 \neq 1$, which
  follows immediately from the definition of $\wshift$. The proof is concluded.
\end{proof}

We define 

\begin{equation}
  \label{eq:minimal}
  m_{\mu}(\alpha) = \min \{m | \wshift^{m+1}(\alpha,\mu) =
  \wshift^{m}(\alpha,\mu)\}.
  \end{equation}

\begin{corollary}
\label{numsteps}
In the special case $\alpha = (p), |\mu|\leq p$ we have 
\[ m_{\mu}(\alpha) = \sum_i(i-1)\mu_i + \frac{(p-|\mu|)(p-|\mu| + 2\ell(\mu)-1)}{2}.\]
\end{corollary}

\begin{proof}
In order to compute $m_{\mu}(\alpha)$, we need to shift the diagram
$(p)$ as many times as we need to obtain a column shape, remembering
that whenever we obtain a shape $\beta$ such that $\mu_i = \max_k\
\beta_k$, we erase the longest row (which we call reduction) and then we apply shift operator to
a new shape. In this case, this longest row is the first row of
$\beta$, which is a direct consequence of the proof of
\Cref{lem:ShiftTerminates}. Consider a tableau of shape $\alpha$
filled by numbers in a way that all the entries in $i$-th row are
$i-1$. Notice that the difference between the sum of the contents of this
tableau of shape $\shift \alpha$ and the sum of the contents of this
tableau of shape $\alpha$ is equal to $1$. In particular, since we were erasing (during reduction) rows of length $\mu_i$ filled by $i-1$, we obtain at the end a column of length $p-|\mu|$ filled by consecutive entries starting from $\ell(\mu)$ (we performed reduction precisely $\ell(\mu)$ times). Therefore
\begin{multline*} 
m_{\mu}(\alpha) = \sum_i(i-1)\mu_i + \sum_{1 \leq i \leq p-|\mu|}(\ell(\mu)+i-1) \\
= \sum_i(i-1)\mu_i + \binom{p-|\mu|+\ell(\mu)}{2} - \binom{\ell(\mu)}{2} \\
= \sum_i(i-1)\mu_i +\frac{(p-|\mu|)(p-|\mu| + 2\ell(\mu)-1)}{2}.
\end{multline*}
\end{proof}

Finally, define a local shift operator 
\[ \localshift: \Comp_l^+ \cup \Comp_l \rightarrow \Comp_l^+ \cup \Comp_{l+1}\]
by shifting the split box, if it exists, onto the next column if there is a next column (hence preserving
the augmented shape),
and by replacing the split box by a normal box and putting another box to its right otherwise. For a composition $\alpha \in \Comp_l$, we define $\localshift(\alpha)$  as the augmented composition obtained by removing the rightmost upper box from the diagram of $\alpha$ and by splitting the first box in the next row.

\begin{lemma}
\label{lemma:locshift}
Let $\alpha \in \Comp_l$ be a unimodal composition, let $j = \min\{i \mid \alpha_i = \max_k \alpha_k\}$ and $r = \alpha_{j+1}$. Then 

\[ \shift(\alpha) = \localshift^{r+1}(\alpha).\]

\end{lemma}

\begin{proof}
By definition, $\localshift(\alpha)$ is an augmentation of $\alpha - e_{j}$, where 
$e_j = (\underbrace{0,\dots,0}_{\text{$j-1$ times}},1,0\dots)$
and $j = \min\{i \mid \alpha_i = \max_k \alpha_k\}$. 
Then the augmented boxes of $\localshift^{r}(\alpha)$ will lie precisely in the last column and in row $j+1$. Therefore 

\[\localshift^{r+1}(\alpha) = \alpha -e_{j}+e_{j+1} = \shift(\alpha),\]

as desired. 
\end{proof}

\begin{exa}
$
\Yboxdim{1cm}
\localshift^3
\begin{tikzpicture}[font=\tiny, scale=0.3,baseline=0cm]
\tgyoung(0cm,1cm,<>,<><><><>,<><>)
\end{tikzpicture}
=
\localshift^2
\begin{tikzpicture}[font=\tiny, scale=0.3,baseline=0cm]
\tgyoung(0cm,1cm,<>,<><><>,<><>)
\draw (0,-1) -- (1,-0);
\end{tikzpicture}
=
\localshift
\begin{tikzpicture}[font=\tiny, scale=0.3,baseline=0cm]
\tgyoung(0cm,1cm,<>,<><><>,<><>)
\draw (1,-1) -- (2,-0);
\end{tikzpicture}
=
\begin{tikzpicture}[font=\tiny, scale=0.3,baseline=0cm]
\tgyoung(0cm,1cm,<>,<><><>,<><><>)
\end{tikzpicture}.
$
\end{exa}

Just as is the case of $\shift$, the map $\localshift$ 
naturally induces a map on augmented natural tableaux, which we denote by the same symbol.

\subsection{Cocyclage in type $\AAA_{n-1}$}
\label{sec:CyclageA}

The \textit{north-eastern column word} 
$w(T)$ of a tableau $T$ is obtained from $T$ by reading its entries, column-wise, from right to left and top to bottom.
In the rest of this section, fix $n\in\Z_{\geq 0}$ 
and consider the type $\AAA_{n-1}$ alphabet $\cA_n=\{ 1,\dots,n \}$.
Following \cite{LascouxSchutzenberger1978}, we define the \textit{cocyclage} of semistandard Young tableau as follows.
Let $T$ be a semistandard Young tableau such that no letter of $\cA_n$ appears in all columns.
In this case, we say that the cocyclage is \textit{authorized} for $T$.
We set $\Cyc_{\AAA} (T) = x \rightarrow T'$, where $T'$ is the unique semistandard Young tableau such that $w(T') \equiv u$ and $w(T) = xu$ for a word $u$ and a letter $x \neq 1$, and where $\equiv$ 
is the congruence relation generated by the plactic relations, see \cite{Lothaire2002}, 
and $*\to U$ is the column Schensted insertion of the letter $*\in\cA$ into the semistandard Young tableau $U$.

\begin{exa}
Let $n=5$ and 
$T =  
\begin{tikzpicture}[font=\tiny, scale=0.3,baseline=0cm]
\tgyoung(0cm,1cm,<1><1><2>,<3><5>,<4>)
\end{tikzpicture}
$. 
Then $w(T) = 2 1 5 1 3 4$, so we take $u = 1 5 1 3 4 $ and $x = 2$. 
We have that $u = w(T')$ where 
$T' =  
\begin{tikzpicture}[font=\tiny, scale=0.3,baseline=0cm]
\tgyoung(0cm,1cm,<1><1>,<3><5>,<4>)
\end{tikzpicture}
$, hence the cocyclage of $T$ is the tableau 
\[\Cyc_{\AAA} (T) = 2 \rightarrow T' = 
\begin{tikzpicture}[font=\tiny, scale=0.3,baseline=0cm]
\tgyoung(0cm,1cm,<1><1><5>,<2><3>,<4>)
\end{tikzpicture}.
\]
\end{exa}

Now we can define cocyclage more generally. 
Let $T$ be a semistandard Young tableau whose weight is not equal to its shape. 
If there is a letter $\ell$ of $\cA_n$ contained in every column of  $T$, we say that the cocyclage is {not authorized} for $T$,
and we define the \textit{reduction} of $T$ to be the tableau $\red(T)$ obtained by deleting (recursively for every such $\ell$), 
all occurences of $\ell$ and replacing all $i>\ell$ by $i-1$.
Then the cocyclage is authorized for $\red(T)$ and we define $\Cyc_{\AAA}(T)=\Cyc_{\AAA}(\red(T))$.

\begin{exa}
\label{ex:wazne2}
The cocyclage is not authorized for the tableau 
$T=  
\begin{tikzpicture}[font=\tiny, scale=0.3,baseline=0cm]
\tgyoung(0cm,1cm,<1><1><1>,<2><2><3>,<4>)
\end{tikzpicture}.
$
We compute 
$\red(T)=
\begin{tikzpicture}[font=\tiny, scale=0.3,baseline=0cm]
\tgyoung(0cm,1cm,<1><1><2>,<3>)
\end{tikzpicture}$
and we get $\Cyc_{\AAA}(T)=\Cyc_{\AAA}(\red(T))=
\begin{tikzpicture}[font=\tiny, scale=0.3,baseline=0cm]
\tgyoung(0cm,1cm,<1><1>,<2><3>)
\end{tikzpicture}.
$
\end{exa}

\begin{rem}\label{rem_red_simp}
Let $\la,\mu$ be two partitions of the same size, and let $T \in \SSYTab(\la,\mu)$.
Note that $\op{simp}$ has the following interpretation:
$\op{simp}(\la,\mu)_1$ is the shape of
$\red(T)$ and $\op{simp}(\la,\mu)_2$ is its weight (see \Cref{alg:red}).
\end{rem}

A quick comparison of \cref{ex:wazne2} and \cref {ex:wazne} suggests
that $\wshift$ corresponds to $\Cyc_{\AAA}$. This is indeed the
case for natural tableaux (modulo gravity).
Although the proof
is easy, it seems that this simple description of cocyclage was
overlooked in the literature. Moreover, it will link the cocylage in type
$\AAA$ with the cocyclage in type $\CC$
as we will show in \Cref{insandshift} (see also \Cref{rem_philo_2} and
\Cref{rem_philo_3}).

\begin{prop}
\label{typeacyc}
Let $T$ be a natural tableau $T \in
\SSTab(\alpha,\mu)$ where $\al$ is a unimodal composition and $\mu$ a partition. Then
\[ \Cyc_{\AAA}(\grav(T)) = \grav(\wshift(T)).\]
\end{prop}

\begin{proof}
First assume that the cocyclage is authorized for $T$.  
Let $\square_{a}, \square_{a+1}$ be consecutive boxes in
$\mathcal{D}_{\alpha}$ with respect to the natural order, with $k =
T(\square_{a}), \ell = T(\square_{a+1})$. Let $C'$ be the column of $T$
containing $\square_{a+1}$ and let $C = \grav C '$. Then 
\begin{align*}
k \rightarrow C = D\Skew(0: \ell ),
\end{align*}
where $D$ is obtained from $C$ by replacing the entry $T(\square_{a+1}) =\ell$ by $k$. 
Since this property only depends on the relative position of the entries in $T$, it follows by induction on the number of
columns that  
\begin{align*}
\grav (\localshift^{r+1} (T) ) = \Cyc_{\AAA}(\grav(T)).
\end{align*}
where $r = \alpha_{j+1}$ and $j = \min\{i \mid \alpha_i = \max_k \alpha_k\}$. 
By \Cref{lemma:locshift} we have
$\localshift^{r+1} (T)  = \shift(T)$,
which yields
$$\Cyc_{\AAA}(\grav(T)) = \grav(\shift(T)).$$
Now, since cocyclage is authorized for $T$, this means that we do not
have to use reduction,
and therefore by \Cref{rem_red_simp} we simply have $\wshift(T)=\shift(T)$.
This finishes the proof in this case.

Assume now that the cocyclage is not authorized for $T$. 
Then there exists some letter $\ell \in \mathcal{A}_{n}$ appearing in each column of $T$.
We have $\max \alpha_{k}  = \mu_\ell$, since the same number can only appear once in each column (since $T$ is semistandard).   
Since $\mu$ is a partition,  this implies 
$\mu_{1} = ... = \mu_{\ell}$ and $\al_1=\ldots=\al_\ell=\mu_\ell=\max \al_k$.
Therefore, since $\al$ is unimodal, $\al$ is a partition. This gives
\begin{align*}
\Cyc_{\AAA} (\grav(T))
& =  \Cyc_{\AAA} (T) \text{ \quad since $\al$ is a partition}
\\
& =  \Cyc_{\AAA} (\red(T)) \text{ \quad since cocyclage is not authorized for $T$}
\\
& = \grav(\shift(\red(T))) \text{ \quad by the previous case}
\\
& = \grav(\wshift(T)) \text{ \quad by \Cref{rem_red_simp}.}
\end{align*}
\end{proof}

\section{Lecouvey's conjecture, symplectic insertion and cocyclage}
\label{typec}
\subsection{Kostka--Foulkes polynomials}
\label{kostka}

Let $\Phi$ be a finite, reduced root system and $\Phi^+ \subset \Phi$ a choice of 
positive roots. We denote by $W$ the corresponding Weyl
group. Similarly, let
$\Lambda$ be the integral weight lattice and $\Lambda^+$ its dominant
part. Let $\Z[\Lambda] =
\Span_\Z\{e^\lambda:\lambda \in \Lambda\}$ denote the group ring of
$\Lambda$. We denote by $\epsilon: \Z[\Lambda] \to \Z[\Lambda]$ the
skew-symmetrizing operator, that is
\[ \epsilon(f) = \sum_{w \in W}(-1)^{\ell(w)}w(f),\]
where $f \in \Z[\Lambda]$. We also recall the the definition of the Weyl character:
\[ \chi(\lambda) = \frac{\epsilon(e^{\lambda +
      \rho})}{\epsilon(e^\rho)},\]
where $\lambda \in \Lambda^+$ is dominant and $\rho =
\frac{1}{2}\sum_{\alpha \in \Phi^+}\alpha$. This is the character of an irreducible $\mathfrak{g}$--module of highest
weight $\lambda$, where $\mathfrak{g}$ is the complex semisimple Lie algebra
associated with $\Phi$.
The Hall--Littlewood polynomial $P_\lambda(q)$ is a one-parameter deformation between
Weyl characters and orbit sums $m(\lambda) = |W_\lambda|^{-1}\sum_{w
  \in W}e^{w(\lambda)}$, where $W_\lambda < W$ is the stabilizer
of $\lambda$.
Indeed,
\[ P_\lambda(q) =
  \epsilon\left(e^{\lambda+\rho}\prod_{\alpha\in\Phi: \langle
      \lambda,\alpha\rangle
      >0}(1-qe^{\alpha})\right)/\epsilon(e^{\rho})\]
and $P_\lambda(0) = \chi(\lambda)$ is the Weyl character while
$P_\lambda(1) = m(\lambda)$ is the orbit sum.

The Kostka--Foulkes polynomials $K_{\lambda,\mu}(q) \in \Z[q]$ for $\lambda,\mu
\in \Lambda^+$ are then defined as the coefficients in the decomposition of the Weyl characters
in the basis of Hall--Littlewood polynomials:

\begin{equation}
\label{eq:KostkaFoulkes}
\chi(\lambda) = \sum_{\mu \in \Lambda^+}K_{\lambda,\mu}(q)P_\mu(q).
\end{equation}

Note that $K_{\lambda,\mu}(1)$ is the dimension of the $\mu$-weight
space of an irreducible $\mathfrak{g}$--module of highest weight
$\lambda$. Moreover, it was conjectured by Lusztig \cite{Lusztig1983}
and proven by Kato \cite{Kato1982} that Kostka--Foulkes polynomials are
appropriately normalized Kazhdan--Lusztig polynomials. This implies
that $K_{\lambda,\mu}(q) \in \Z_{\geq 0}[q]$ has nonnegative integer
coefficients, which naturally leads to \Cref{problem}.

In the following we are going to investigate \Cref{problem} when
$\Phi$ is the root system of type $\CC_n$.
We will use the superscript $\CC_n$ to indicate that we work in this case.

\subsection{Symplectic tableaux}
\label{symptab}
Let $n$ be a positive integer and $\lambda,\mu$ partitions with at most $n$
parts. From now on, $\mathfrak{g} = \mathfrak{sp}_{2n}(\mathbb{C})$ will be the complex symplectic Lie algebra, whose associated root system is of type $\CC_{n}$.  
A  \textit{Kashiwara--Nakashima} tableau, or
\textit{symplectic} tableau of shape $\lambda$ and
\textit{weight} $\mu$ is a Young tableau 
\[T \in \bigcup_{\beta}\SSYTab_{\mathcal{C}_n}(\lambda,\beta),\]
such that
\begin{itemize}
\item
$\mathcal{C}_n = \{\ov{n}< \dots < \ov{1} < 1 \dots < n\},$
\item
  we take the union over $\beta$ of the form 
  \[\beta =
(k_n+\mu_{\ov{n}},k_{n-1}+\mu_{\ov{n-1}},\dots,
k_{1}+\mu_{\ov{1}},k_1,\dots,k_n),\]
where $k_1,\dots,k_n \in \Z_{\geq
    0}$ and $\mu = (\mu_{\ov{n}},\dots, \mu_{\ov{1}})$,
\item
each one of its columns is \textit{admissible},
\item
The \textit{split version} of $T$ is semistandard.
\end{itemize}
The last two conditions will not be used in this work, therefore we
refer the reader to \cite{Lecouvey2005} for a detailed definition.
Given partitions $\mu,\lambda$ we will denote the set of symplectic tableaux of shape $\lambda$ and weight $\mu$ by $\op{SympTab}_{n}(\lambda,\mu)$. 
The following proposition justifies why we do not need the last two
defining properties of symplectic tableaux.

\begin{prop/def}
\label{prop/def}
Let $\lambda = (p)$ and $\mu$ be a partition. Then 
\[\op{SympTab}_{n}(\lambda, \mu) = \bigcup_{k_1,\dots,k_n \in \Z_{\geq
    0}}\SSYTab_{\mathcal{C}_n}(\lambda, (k_n+\mu_{\ov{n}},k_{n-1}+\mu_{\ov{n-1}},\dots,
  k_{1}+\mu_{\ov{1}},k_1,\dots,k_n)). \]
\end{prop/def}

We will also use the following notation:

\[\mathcal{C}= \underset{n \in \mathbb{Z}_{\geq 1}}{\bigcup} \mathcal{C}_{n} = \{\dots  < \overline{n} < \dots < \overline{1} < 1 <
  \dots \overline{n} < \dots \}, \]

\noindent
with the convention that $\ov{\ov{n}} = n$ and 

\[\op{SympTab}_{n}(\lambda) = \underset{\mu}{\bigcup}
  \op{SympTab}_{n}(\lambda,\mu), \qquad \op{SympTab}_{n} = \underset{\lambda}{\bigcup} \op{SympTab}_{n}(\lambda).\]

For two integers $i \leq j $, we will use the following notation:

\[ [i,j]_{\mathcal{C}} : = \{ k \in [i,j]: k \neq 0 \} \] 

where

\[ [i,j] = \{k \in \mathbb{Z}| i \leq k \leq j \}. \]

We are interested in the set of symplectic tableaux since these
objects give a natural basis of the $\mu$-weight
space of an irreducible $\mathfrak{g}$--module of highest weight
$\lambda$ in type $\CC$, see \cite{KashiwaraNakashima1994}. 
Therefore

\[ K^{\CC_n}_{\lambda,\mu}(1) = |\op{SympTab}_{n}(\lambda,\mu)|.\]

\subsection{Symplectic insertion}

We recall the definition of symplectic insertion as introduced in \cite{Lecouvey2005}. 
Given a letter $* \in \mathcal{C}$ and an admissible column $C$
(again, we do not really need the definition of admissibility in this
work, but roughly speaking this is a condition which assures that the
insertion $*\to C$ described in the following part produces a
symplectic tableau, see \cite{Lecouvey2005}), the insertion $* \rightarrow C$ is defined as follows. 
If $*$ is larger than all the letters of $C$, then place it in a new box at the bottom of $C$.
This yields a column $C'$ and we set $*\to C=C'$.
Otherwise,
if $C = \Skew(0:\hbox{\scriptsize{$a$}})$ consists of only one box, set $$*\rightarrow C \coloneqq \Skew(0:\hbox{\scriptsize{$*$}},\hbox{\scriptsize{$a$}}).$$ 
The insertion of a letter into a column of length at least 2 is defined inductively as follows.
For the base case, assume that 
$C = \Skew(0:\hbox{\scriptsize{$a$}}|0:\hbox{\scriptsize{$b$}})$ consists of two boxes. Then we consider the following four cases: 
\begin{enumerate}[label=(I$\arabic*$), ref=(I\arabic*)]
\item \label{I1} If $a<* \leq b$ and $b\neq \overline a$, then 
$$* \rightarrow \Skew(0:\hbox{\scriptsize{$a$}}|0:\hbox{\scriptsize{$b$}}) \coloneqq 
\op{grav}
\Skew(
0:\hbox{\scriptsize{$a$}}
|0:\hbox{\scriptsize{$*$}},\hbox{\scriptsize{$b$}}
).
$$
\item \label{I2} If $*\leq a < b$ and $b\neq \overline *$, then $$*
  \rightarrow
  \Skew(0:\hbox{\scriptsize{$a$}}|0:\hbox{\scriptsize{$b$}}) \coloneqq
  \grav\Skew(0:\hbox{\scriptsize{$*$}},\hbox{\scriptsize{$a$}}|0:\hbox{\scriptsize{$b$}}).$$
\item \label{I3} If $a = \overline b$ and $\overline b \leq * \leq b$, 
then $$* \rightarrow \Skew(0:\hbox{\scriptsize{$\overline{b}$}}|0:\hbox{\scriptsize{$b$}}) \coloneqq
\op{grav}
\Skew(
0:\hbox{\scalebox{.6}{$\ov{b\text{+}1}$}}
|0:\hbox{\scriptsize{$*$}},\hbox{\scalebox{.6}{$b$+$1$}}
).$$
\item \label{I4} If $* = \overline b$ and $\overline b < a < b$, then $$* \rightarrow \Skew(0:\hbox{\scriptsize{$a$}}|0:\hbox{\scriptsize{$b$}}) \coloneqq \grav
\Skew(0:\hbox{\scalebox{.6}{$\ov{b\text{-}1}$}},\hbox{\scriptsize{$a$}}
|0:\hbox{\scalebox{.6}{$b\text{-}1$}}
).$$ 
\end{enumerate}
Note that cases \ref{I1} and \ref{I2} amount to ordinary column
bumping.

Let $C$ be of length $k\geq3$, and suppose that the insertion of a letter
into a column of length $k-1$ has been already defined and yields an $n$-symplectic tableau
of shape $(2,1^{k-2})$.
Write $C=\Skew(
0:\hbox{\scalebox{.8}{$a_1$}}
|0:\hbox{\scalebox{.8}{$a_2$}}
|0:\hbox{\scalebox{.8}{$\vdots$}}
|0:\hbox{\scalebox{.8}{$a_k$}}
)
$ 
and $C'=
\Skew(
0:\hbox{\scalebox{.8}{$a_2$}}
|0:\hbox{\scalebox{.8}{$\vdots$}}
|0:\hbox{\scalebox{.8}{$a_k$}}
)$.
Let
$*\to C'=
\Skew(
0:\hbox{\scalebox{.8}{$\be_2$}},\hbox{\scalebox{.8}{$y$}}
|0:\hbox{\scalebox{.8}{$b_3$}}
|0:\hbox{\scalebox{.8}{$\vdots$}}
|0:\hbox{\scalebox{.8}{$b_k$}}
)
$ and 
$\be_2\to
\Skew(
0:\hbox{\scalebox{.8}{$a_1$}}
|0:\hbox{\scalebox{.8}{$y$}}
)
=
\Skew(
0:\hbox{\scalebox{.8}{$b_1$}},\hbox{\scalebox{.8}{$z$}}
|0:\hbox{\scalebox{.8}{$b_2$}}
)
$.
Then
$*\to C \coloneqq
\Skew(
0:\hbox{\scalebox{.8}{$b_1$}},\hbox{\scalebox{.8}{$z$}}
|0:\hbox{\scalebox{.8}{$\vdots$}}
|0:\hbox{\scalebox{.8}{$b_{\scalebox{.6}{$k$-$1$}}$}}
|0:\hbox{\scalebox{.8}{$b_{k}$}}
),
$
which is a symplectic tableau.

\begin{exa} Take $*=\ov{3}$ and 
$
C=
\Skew(
0:\hbox{\scalebox{.8}{$\ov{5}$}}
|0:\hbox{\scalebox{.8}{$\ov{3}$}}
|0:\hbox{\scalebox{.8}{$\ov{1}$}}
|0:\hbox{\scalebox{.8}{$3$}}
).$
We first need to compute 
$\ov{3}\to 
\Skew(
0:\hbox{\scalebox{.8}{$\ov{3}$}}
|0:\hbox{\scalebox{.8}{$\ov{1}$}}
|0:\hbox{\scalebox{.8}{$3$}}
)$.
For this we compute
$\ov{3}\to 
\Skew(
0:\hbox{\scalebox{.8}{$\ov{1}$}}
|0:\hbox{\scalebox{.8}{$3$}}
)=
\Skew(
0:\hbox{\scalebox{.8}{$\ov{2}$}},\hbox{\scalebox{.8}{$\ov{1}$}}
|0:\hbox{\scalebox{.8}{$2$}}
)
$
\quad and
$\ov{2}\to 
\Skew(
0:\hbox{\scalebox{.8}{$\ov{3}$}}
|0:\hbox{\scalebox{.8}{$\ov{1}$}}
)=
\grav
\Skew(
0:\hbox{\scalebox{.8}{$\ov{3}$}}
|0:\hbox{\scalebox{.8}{$\ov{2}$}},\hbox{\scalebox{.8}{$\ov{1}$}}
)=
\Skew(
0:\hbox{\scalebox{.8}{$\ov{3}$}},\hbox{\scalebox{.8}{$\ov{1}$}}
|0:\hbox{\scalebox{.8}{$\ov{2}$}}
),
$
\quad and we get
$\ov{3}\to 
\Skew(
0:\hbox{\scalebox{.8}{$\ov{3}$}}
|0:\hbox{\scalebox{.8}{$\ov{1}$}}
|0:\hbox{\scalebox{.8}{$3$}}
)=
\Skew(
0:\hbox{\scalebox{.8}{$\ov{3}$}},\hbox{\scalebox{.8}{$\ov{1}$}}
|0:\hbox{\scalebox{.8}{$\ov{2}$}}
|0:\hbox{\scalebox{.8}{$2$}}
)
$.
Finally, since
$\ov{3}\to 
\Skew(
0:\hbox{\scalebox{.8}{$\ov{5}$}}
|0:\hbox{\scalebox{.8}{$\ov{1}$}}
)=
\grav
\Skew(
0:\hbox{\scalebox{.8}{$\ov{5}$}}
|0:\hbox{\scalebox{.8}{$\ov{3}$}},\hbox{\scalebox{.8}{$\ov{1}$}}
)=
\Skew(
0:\hbox{\scalebox{.8}{$\ov{5}$}},\hbox{\scalebox{.8}{$\ov{1}$}}
|0:\hbox{\scalebox{.8}{$\ov{3}$}}
),
$
\quad
we get
$$*\to C=
\Skew(
0:\hbox{\scalebox{.8}{$\ov{5}$}},\hbox{\scalebox{.8}{$\ov{1}$}}
|0:\hbox{\scalebox{.8}{$\ov{3}$}}
|0:\hbox{\scalebox{.8}{$\ov{2}$}}
|0:\hbox{\scalebox{.8}{$2$}}
)
.
$$
\end{exa}

\begin{figure}[t]
    \centering
    \includegraphics[width=\linewidth]{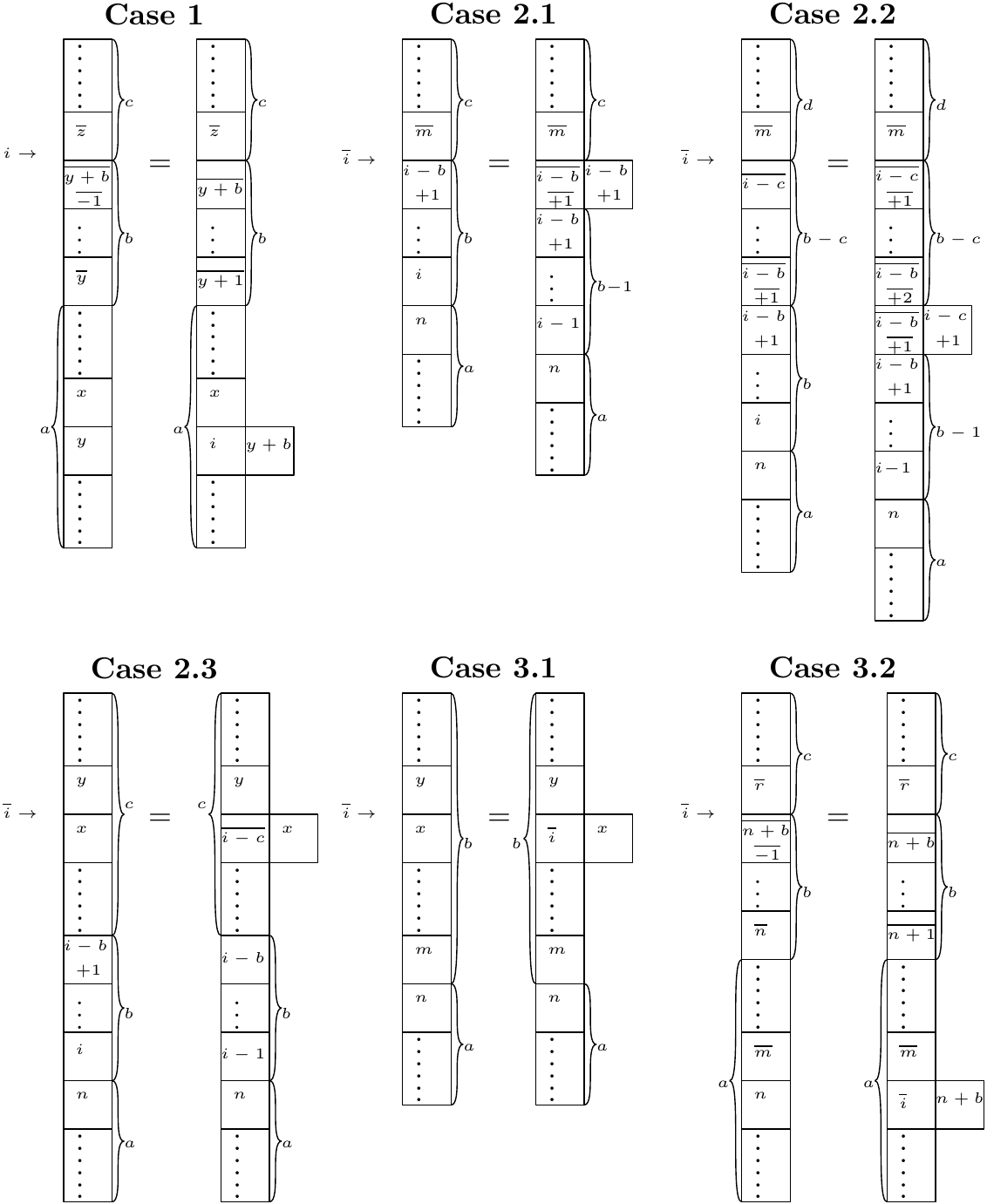}
\caption{All the possible cases in the symplectic insertion $* \to C$
  described by \cref{prop:Insertion}.}
\label{fig:1p1}
\end{figure}

The above definition is not very helpful in practice. 
Indeed, we would like to understand the global impact of 
inserting a letter into a column, 
while the nature of presented definition is local and recursive. 
The following proposition lets us overcome this difficulty.

\begin{proposition}
\label{prop:Insertion}
Let $C$ be a column, that is, a Young tableau of shape $(1, \dots ,
1)$, and let $*$ be an entry not larger than the maximal entry of
$C$. The insertion $* \rightarrow C$ can be classified into three
cases depending on whether $*$ is barred, and whether $\overline{*}$ belongs to $C$. These cases (and corresponding
subcases) amounts
to performing the operations presented in \cref{fig:1p1}, followed by
applying $\op{grav}$, where parameters $*,a,b,c,d,m,n,r,x,y,z$ are
described below:

\begin{enumerate}[label=$\bullet$, ref=Case \arabic*, 
]
\item \label{1}  \textbf{Case 1} from~\cref{fig:1p1} when $* = i$ is unbarred, and
\begin{itemize}[label=$\circ$]
\item
$a \geq 1, b \geq 0, x< i \leq y$, (in the case $a=1$ column $C$ necessarily
contains  $y$)
\item
If $b>0$, $c \geq 0, x < i \leq y < z-b$ (whenever $b = 0$, $c$ and $z$ are not defined). 
\end{itemize}
\item \label{2} \textbf{Case 2}. When $* = \overline i$ is barred and $i
  \in C$ we have the following subcases:
\begin{enumerate}[label=$\bullet$, ref=Case 2.\arabic*, leftmargin=0cm]
\item \label{2.1} \textbf{Case 2.1} from~\cref{fig:1p1} with
\begin{itemize}[label=$\circ$]
\item
$a \geq 0, 1 \leq b \leq i, c \geq 0$,
\item
$n> i$,
\item
$m> i-b+1$ (defined whenever $c > 0$).
\end{itemize}
\item \label{2.2}\textbf{Case 2.2} from~\cref{fig:1p1} with
\begin{itemize}[label=$\circ$]
\item
$a \geq 0, 1 \leq b \leq i, b-c > 0, d \geq 0$,
\item
$n> i$,
\item
$m > i-c+1$.
\end{itemize}
\item \label{2.3} \textbf{Case 2.3} from~\cref{fig:1p1} with
\begin{itemize}[label=$\circ$]
\item
$a \geq 0, 1 \leq b \leq i, c \geq 1$ ($C$ necessarily contains $x$),
\item
$y < \overline{i-b+1} \leq x$, with the condition that 
there is a box between $x$ and $i-b+1$ if $\overline{i-b+1} = x$,
\item
$n> i$.
\end{itemize}
\end{enumerate}
\item \label{3} \textbf{Case 3}. When $* = \overline i$ is barred and $i
  \notin C$ we have the following subcases:
\begin{enumerate}[label=$\bullet$, ref=Case 3.\arabic*, leftmargin=0cm]
\item \label{3.1} \textbf{Case 3.1} from~\cref{fig:1p1} with
\begin{itemize}[label=$\circ$]
\item
$a \geq 0$, $b \geq 1$  ($C$ necessarily contains $x$) 
\item
$n > i  > m $,
\item
$y < \overline{i} \leq x$.
\end{itemize}
\item \label{3.2}\textbf{Case 3.2} from~\cref{fig:1p1} with
\begin{itemize}[label=$\circ$]
\item
$a \geq 1,$ ($C$ necessarily contains $n$), $b, c \geq 0$,
\item
$n > m  > i$, with the possibility that $\ov{m}$  or $\ov{n}$ do not appear in $C$ (whenever $a=1$ or $b=0$, respectively)
\item
$r > n+b$, whenever $b>0$.
\end{itemize}
\end{enumerate}
\end{enumerate}
\end{proposition}

\begin{figure}[t]
    \centering
    \includegraphics[width=\linewidth]{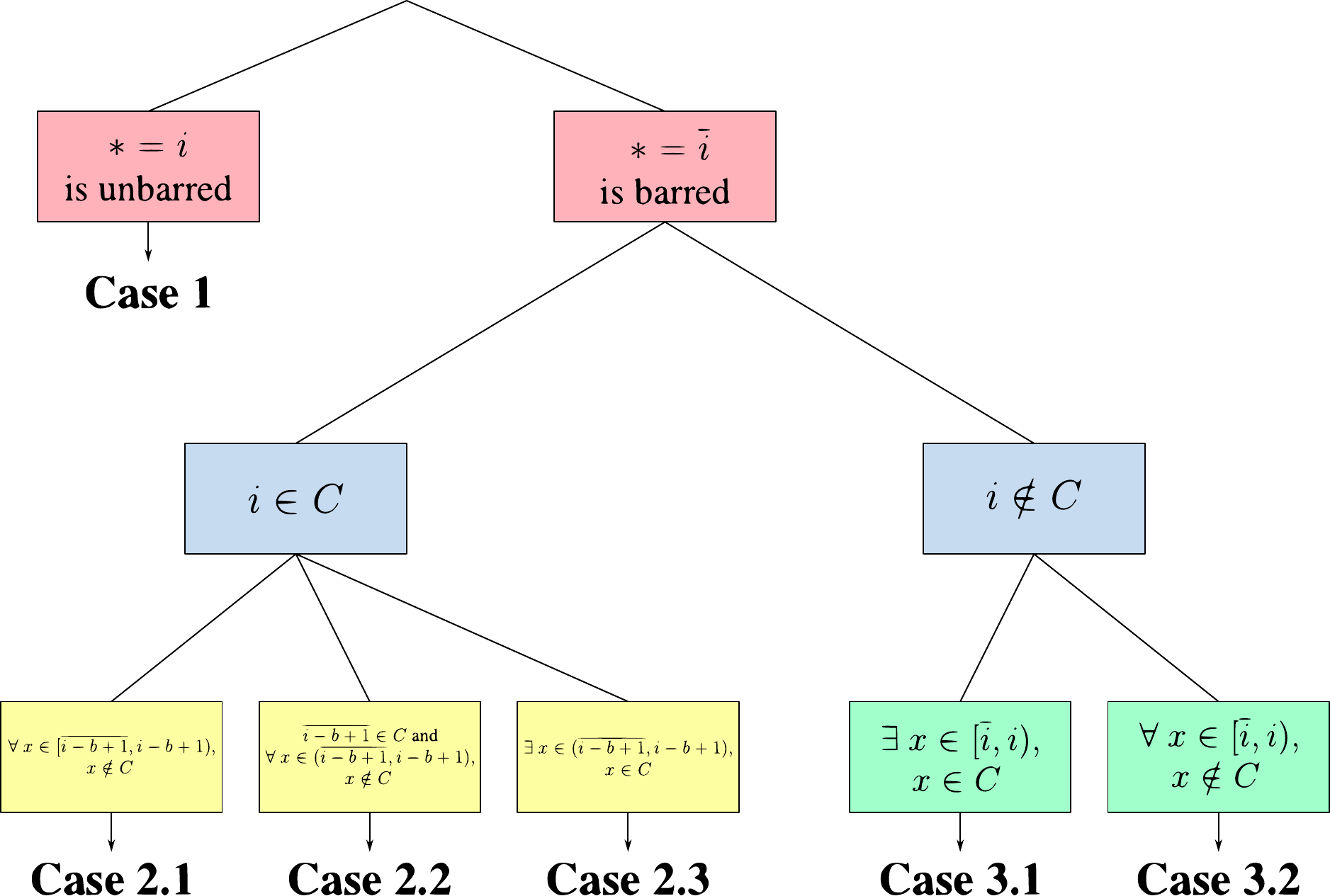}
\caption{}
\label{fig:tree}
\end{figure}

\begin{proof}
By searching the tree presented on \Cref{fig:tree}, we are ensured that
we are always in \ref{1} -- \ref{3} and that all the cases are pairwise
distinct.
We prove the formulas of \ref{1} -- \ref{3} by induction on the length $\ell$ of $C$.
In the case of columns of length at most $2$, this description coincides with the original definition. 
Fix $\ell>2$, assume that the claim holds for all columns
of length $\ell-1$ and let $C$ be a column of length $\ell$. 
Let $C'$ be a column obtained from $C$ by removing its top box $\Skew(0:\hbox{\scriptsize{$t$}})$.
By definition, $* \to C$ is obtained by first performing 
$* \to C'=C''\Skew(0:\hbox{\scriptsize{$t'$}})$ and then inserting the
top entry of $C''$ into $\Skew(0:\hbox{\scriptsize{$t$}}|0:\hbox{\scriptsize{$t'$}})$.
Since the analysis of all the cases is very similar, we only show the proof of \ref{1} and \ref{2.2} 
(where all the possible difficulties are present), 
leaving the proof of the other cases as an easy exercise. 

\hspace{10pt}

\textbf{\ref{1}}. We have either $c >0$ or $c=0$. In the case $c>0$, performing $* \to C'$
yields the shape $C''\Skew(0:\hbox{\scalebox{.7}{$y$+$b$}})$ 
described by \ref{1}, by induction hypothesis.
Then we have to insert the top entry $u$ of $C''$ (which is either the
top entry of $C'$ in the case $c > 1$ or is equal to $\ov{y+b}$) 
into the column $\Skew(0:\hbox{\scriptsize{$t$}}|0:\hbox{\scalebox{.7}{$y$+$b$}})$. 
Since we have $t<u<y+b$, we need to apply the local insertion rule \ref{I1},
which yields the shape described by \ref{1}. 
In the case $c=0$, we have either $b>0$ or $b=0$. 
Suppose first $b = 0$. Then either $y$ is the top entry of $C$, the second
entry from the top, or the $k$-th entry from the top with $k>2$. 
In the first case, we have $t=y$.
Therefore, by induction hypothesis, $i \to C'=C''\Skew(0:\hbox{\scriptsize{$t'$}})$ where the top entry of $C''$ is $i$. 
Thus, it remains to insert $i$ into $\Skew(0:\hbox{\scriptsize{$y$}}|0:\hbox{\scriptsize{$t'$}})$, 
which, by the local insertion rule \ref{I2}, simply bumps out $y$ since $i\leq y$. 
In the second case, by induction hypothesis, after performing $i \to C'$ we
have to insert $i$ to the column $\Skew(0:\hbox{\scriptsize{$x$}}|0:\hbox{\scriptsize{$y$}})$, 
which bumps out $y$ by the local insertion rule \ref{I1}. 
In the last case, by induction hypothesis, after performing $i \to C'$, we
have to insert the top entry $u$ of $C''$, which coincides with the top
entry of $C'$ and satisfies $t<u<y$, into the column
$\Skew(0:\hbox{\scriptsize{$t$}}|0:\hbox{\scriptsize{$y$}})$. Here
again we apply
the local insertion rule \ref{I1}, which amounts to bumping out $y$. 
In all three configurations, this yields the shape described by \ref{1}.
Finally, suppose that $b > 0$. 
By induction hypothesis, after performing $i \to C'$,
we have to insert the top entry $u$ of $C''$, which coincides with the top
entry of $C'$ and satisfies $\overline{y+b-1}<u<y+b-1$, into the column
$\Skew(0:\hbox{\scalebox{.5}{$\overline{y\text{+}b\text{-}1}$}}|0:\hbox{\scalebox{.5}{$y$+$b$-$1$}})$.
Here we apply
the local insertion rule \ref{I3}, which gives
$\grav \Skew(0:\hbox{\scalebox{.7}{$\overline{y\text{+}b}$}}|0:\hbox{\scriptsize{$u$}},\hbox{\scalebox{.7}{$y$+$b$}})$.
Once again, this yields the shape described in \ref{1}.

\hspace{10pt}

\textbf{\ref{2.2}}. We have either $d >0$ or $d=0$. 
In the case $d>0$, performing $\ov{i} \to C'$
yields the shape described by \ref{2.2}, by induction hypothesis.
We have to insert the top entry $u$ of $C''$, which coincides with the top
entry of $C'$ and satisfies
$t<u<i-c+1$, into the column
$\Skew(0:\hbox{\scriptsize{$t$}}|0:\hbox{\scalebox{.5}{$i$-$c$+$1$}})$. 
By the local insertion rule \ref{I1}, 
this simply bumps out the entry $i-c+1$, which yields the
shape described in \ref{2.2}. 
In the case $d=0$, we either have $b-c>1$ or $b-c=1$. 
Suppose $b-c>1$. By induction hypothesis, 
after performing $\ov{i} \to C'$, which is described by \ref{2.2},
we have to insert $\overline{i-c}$ to the column
$\Skew(0:\hbox{\scalebox{.7}{$\ov{i\text{-}c}$}}|0:\hbox{\scalebox{.7}{$i$-$c$}})$. 
Here we apply
the local insertion rule \ref{I3}, which gives
$\grav \Skew(0:\hbox{\scalebox{.5}{$\overline{i\text{-}c\text{+}1}$}}
|0:\hbox{\scalebox{.7}{$\overline{i\text{-}c}$}},\hbox{\scalebox{.5}{$i$-$c$+$1$}})$.
Suppose $b-c=1$. By induction hypothesis, $\ov{i}\to C'$ corresponds to \ref{2.1}
with $c=0$. Therefore after performing $\ov{i} \to C'$, we have to insert
$\overline{i-b+1}$ into the column 
$\Skew(0:\hbox{\scalebox{.5}{$\overline{i\text{-}b\text{+}1}$}}|0:\hbox{\scalebox{.5}{$i$-$b$+$1$}})$.
Here again we apply
the local insertion rule \ref{I3}, which yields
$\grav \Skew(0:\hbox{\scalebox{.5}{$\overline{i\text{-}b\text{+}2}$}}
|0:\hbox{\scalebox{.5}{$\overline{i\text{-}b\text{+}1}$}},\hbox{\scalebox{.5}{$i$-$b$+$2$}})=
\grav \Skew(0:\hbox{\scalebox{.5}{$\overline{i\text{-}c\text{+}1}$}}
|0:\hbox{\scalebox{.5}{$\overline{i\text{-}b\text{+}1}$}},\hbox{\scalebox{.5}{$i$-$c$+$1$}})$.
In both cases we obtain \ref{2.2} described in the statement.

The proof of the remaining cases is analogous.
\end{proof}

We can now define the insertion $*\to T$ of a letter $*$ into a symplectic tableau $T$.
This is achieved by the following recursive procedure. Let $T'$
denote the result of inserting $*$ into the first column of $T$ according to the previous rule.
Denote by $T''$ the tableau obtained from $T$ by removing its first column.
If $T'$ is a column, juxtapose 
this column with $T''$.
Otherwise, $T'$ is the juxtaposition of a column and a box 
$\Skew(0:\hbox{\scalebox{.9}{$b$}})$. Then juxtapose this column with $(b\to T'')$. 
It is proved in \cite{Lecouvey2005} that this procedure yields a well-defined map between
$\SympTab_{n}$ and $\SympTab_{n+1}$.

Let $\alpha$ be a unimodal composition and $T \in
\Tab_{\mathcal{C}_n}(\alpha)$ such that $\grav(T) \in \op{SympTab}_n$. We
call such a tableau symplectic of shape $\alpha$. We can
use \Cref{prop:Insertion} to define the insertion $*\to T$ of a letter
$* \in \mathcal{C}_n$. In order to do this, we follow the above
definition of the insertion but additionally recording the vertical
shift between the columns of $T$ and the vertical shift of the box
bumped out. Note that this definition naturally extends the definition
of the insertion to tableaux of partition shape to tableaux of
unimodal composition shape and
$\grav (*\to T) = *\to (\grav T)$.
In particular, the insertion of an entry into an  $n$-symplectic tableau yields
an $n+1$-symplectic tableau.

\begin{exa}\label{exa_symp_insertion}
Let $*=\ov{3}$ and 
$T=\Skew(
0:\hbox{\scriptsize{$\ov{8}$}},\hbox{\scriptsize{$\ov{5}$}}|
0:\hbox{\scriptsize{$\ov{5}$}},\hbox{\scriptsize{$\ov{4}$}}|
0:\hbox{\scriptsize{$\ov{3}$}},\hbox{\scriptsize{$3$}},\hbox{\scriptsize{$8$}}
).
\vspace{2mm}$
The insertion $* \rightarrow  T$ can be computed by successive applications of \Cref{prop:Insertion}.
We have
$$
\begin{array}{cl}
\ov{3} \rightarrow 
\Skew(
0:\hbox{\scriptsize{$\ov{8}$}}|
0:\hbox{\scriptsize{$\ov{5}$}}|
0:\hbox{\scriptsize{$\ov{3}$}}
)
=
\Skew(
0:\hbox{\scriptsize{$\ov{8}$}}|
0:\hbox{\scriptsize{$\ov{5}$}}|
0:\hbox{\scriptsize{$\ov{3}$}},\hbox{\scriptsize{$\ov{3}$}}
)
\vspace{2mm}
&
\text{ by \ref{3.1},}
\\
\ov{3} \rightarrow 
\Skew(
0:\hbox{\scriptsize{$\ov{5}$}}|
0:\hbox{\scriptsize{$\ov{4}$}}|
0:\hbox{\scriptsize{$3$}}
)
=
\Skew(
0:\hbox{\scriptsize{$\ov{5}$}}|
0:\hbox{\scriptsize{$\ov{4}$}}|
0:\hbox{\scriptsize{$\ov{3}$}},\hbox{\scriptsize{$3$}}
)
\vspace{2mm}
&
\text{ by \ref{2.1}, and}
\\
3 \rightarrow 
\Skew(
0:\hbox{\scriptsize{$8$}}
)
=
\Skew(
0:\hbox{\scriptsize{$3$}},\hbox{\scriptsize{$8$}}
)
&
\text{ by \ref{1}.}
\end{array}
$$
Therefore, we get $*\ra T =
\Skew(
0:\hbox{\scriptsize{$\ov{8}$}},\hbox{\scriptsize{$\ov{5}$}}|
0:\hbox{\scriptsize{$\ov{5}$}},\hbox{\scriptsize{$\ov{4}$}}|
0:\hbox{\scriptsize{$\ov{3}$}},\hbox{\scriptsize{$\ov{3}$}},\hbox{\scriptsize{$3$}},\hbox{\scriptsize{$8$}}
).$
\end{exa}

\subsection{Symplectic cocyclage and charge.}
\label{sympchargedef}

Before we describe the statistic $\op{ch}_n$, we need to introduce the type $\CC$ analogue of the cocyclage presented in \Cref{sec:CyclageA}. 
Let $T$ be a symplectic tableau and let $w = w(\grav T)$ be the column
reading word
of the associated Young tableau. If $w=xu$ where $x$ is a letter, it
is readily shown that $u$ is the word of a symplectic tableau $U$,
obtained from $T$ by removing the corresponding box. The cocyclage operation on $w$ is $\eta(w) = ux$. 
The cocyclage operation may or may not be \textit{authorized} for a given symplectic tableau $T$. 
The following result from \cite[4.3]{Lecouvey2005} characterizes this property. 

\begin{prop/def}
\label{authorized}
Let $\mu$ be a partition with at most $n$ parts, and let $T$ be a symplectic tableau of weight $\mu$ with at least two columns. The cocyclage operation
is not authorized on $T$ if and
only if there exists $1\leq p \leq n$ such that $\mu_{\overline{p}}$ equals the
number of columns of $T$ (which is equivalent to the condition that $\mu_{\overline{n}}$ equals the
number of columns of $T$ since
$\mu$ is a partition).
\end{prop/def}

In fact, if $T$ is a symplectic tableau for which the cocyclage operation is not authorized, we can construct from $T$
a symplectic tableau, called the \textit{reduction} $\op{red}(T)$ of $T$, for which the cocyclage is authorized.
Let $t: \mathcal{C} \rightarrow \mathcal{C}$ be the map defined as
follows:
\[ t(c) = \begin{cases} i+1 &\text{ if } c = i,\\
 \ov{i+1} &\text{ if } c = \ov{i}.\\
\end{cases}\]
We define $\op{red}(T)$ of $T$ recursively as follows.
\begin{enumerate}
\item Set $P = T$.
\item  Delete all the $\ov{n}$'s from $P$ and apply $t$ to all entries $x$ of $P$ such that  $\ov{n} < x < n$ to obtain a new (possibly empty) tableau $T'$.
\item  If $T'$ is authorized, then set $\op{red}(T)= T'$. Otherwise,
  set $P = T'$ and go back to the previous step.
\end{enumerate}

\begin{rem}
  \label{rem:reduction}
  Let $T \in \SympTab_n(\alpha,\mu)$. Note that \Cref{alg:red} was
  defined in a way that it mimics steps in reduction of $T$. Therefore
  it is clear that $\red(T) \in \SympTab_n(\op{simp}(\alpha,\mu))$.
\end{rem}

By convention, if the cocyclage operation is authorized on $T$ we set $\op{red}(T) = T$.
By construction, the cocyclage is authorized for $\op{red}(T)$.

\begin{definition}
Let $T \in \op{SympTab}_{n}$ be a symplectic tableau. If $T$ is a column,
we set $\op{CoCyc}_{\CC}(T)= \op{red}T$. Otherwise let $w = xu =
w(\op{red}(T))$, where $x \in \CCC$ and let $U$ be the symplectic tableau with $w(U) = u$. 
Then we define $\op{CoCyc}_{\CC}(T) = \red\big(x \rightarrow U\big)$.
\end{definition}

\begin{exa}
Let $T=
\Skew(
0:\hbox{\scriptsize{$\ov{8}$}},\hbox{\scriptsize{$\ov{5}$}}|
0:\hbox{\scriptsize{$\ov{5}$}},\hbox{\scriptsize{$\ov{4}$}},\hbox{\scriptsize{$\ov{3}$}}|
0:\hbox{\scriptsize{$\ov{3}$}},\hbox{\scriptsize{$3$}},\hbox{\scriptsize{$8$}}
)
\vspace{2mm}$.
Then $\op{CoCyc}_{\CC}(T) =\ov{3} \rightarrow 
\Skew(
0:\hbox{\scriptsize{$\ov{8}$}},\hbox{\scriptsize{$\ov{5}$}}|
0:\hbox{\scriptsize{$\ov{5}$}},\hbox{\scriptsize{$\ov{4}$}}|
0:\hbox{\scriptsize{$\ov{3}$}},\hbox{\scriptsize{$3$}},\hbox{\scriptsize{$8$}}
)
\vspace{2mm},$
which has already been computed in \Cref{exa_symp_insertion}.
We get $\op{CoCyc}_C(T) =
\Skew(
0:\hbox{\scriptsize{$\ov{8}$}},\hbox{\scriptsize{$\ov{5}$}}|
0:\hbox{\scriptsize{$\ov{5}$}},\hbox{\scriptsize{$\ov{4}$}}|
0:\hbox{\scriptsize{$\ov{3}$}},\hbox{\scriptsize{$\ov{3}$}},\hbox{\scriptsize{$3$}},\hbox{\scriptsize{$8$}}
)
\vspace{2mm}.$
\end{exa}

Let $T \in \op{SympTab}_{n}$ be a symplectic tableau. Then there exists a non-negative
integer $m$ such that $\Cyc_{\CC}^{m}(T)$ is a column $C(T)$ of weight
zero \cite[Proposition 4.2.2]{Lecouvey2005}. We denote by $m(T)$ the smallest non-negative integer with this
property. For a symplectic column $C$ of weight zero we set

\[E_{C}= \{ i \geq 1| i \in C, i+1 \notin C\}.\]

\noindent
The charge of $C$ is defined by

\[ \op{ch}_{n}(C) = 2 \underset{i \in E_{C}}{\sum} (n-i),\]

\noindent
and the charge of an arbitrary symplectic tableau $T$ is defined by 

\[\op{ch}_{n}(T) = m(T) + \op{ch}_{n}(C(T)).  \]

\subsection{Breaking down the insertion of a letter/box in a tableau}

In this section we describe the cocyclage $\Cyc_{\CC}$ in terms of augmented
tableaux introduced in \Cref{augmenteddefs}. This description is an important tool
to describe an iterated application of cocyclage as a simple operation
related with an iterated application of cocyclage in type $\AAA$.

Let $\al\vDash l-1$ be unimodal, and let $T \in \Tab^+(\al)$ be an augmented tableau of shape $(\al,b)$
such that $T_+$ has admissible columns and let  $j = T_{-}(b)$.
Write $T_+$ as the concatenation of its columns $T=C_1 C_2\dots C_t$, and let $m$ be such that $b\in C_m$.
We define a map
\newcommand{\locins}{\mathrm{locins}}
$\locins : \Tab^+(\al) \to \Tab^+_{l-1} \sqcup \Tab_{l}$
as follows

\begin{align*}
\locins (T) =
\left\{ 
\begin{array}{ll}
C_1\dots C_{m-1}C_m'C_{m+1}\dots C_t \in \Tab_{l} &\text{ if } j \to C_m = C_m' \text{ is a column, }  \\
&\\
C_1\dots C_{m-1}C_m'C'_{m+1}\dots C_t \in \Tab_{l} 
&\text{ if } j \to C_m = C_m'\Skew(0: \hbox{\tiny{$j'$}}) \text{ is
                                                      not a column} \\
  &\text{ and } j'\to C_{m+1}=C'_{m+1} \text{ is a column,}\\
&\\
T' \in\Tab^+_{l-1}  
&\text{ otherwise, }
\end{array}\right.
\end{align*}
where
\begin{itemize}
  \item[$\bullet$] $T'_+=C_1\dots C_{m-1}C_m'C_{m+1}\dots C_t$,
 \item[$\bullet$] $T'$ has shape $(\al,b')$ with
$b'=(m+1,-r)\in\cD_{\al}$, 
 \item[$\bullet$] $r$ is the row of $\Skew(0:
\hbox{\tiny{$j''$}})$ in $j'\to C_{m+1}  = C'_{m+1}\Skew(0:
\hbox{\tiny{$j''$}})$, where $j \to C_m = C_m'\Skew(0: \hbox{\tiny{$j'$}})$,
 \item[$\bullet$] $T'_-(b')=j'$ (which determines $T'$ by Remark \ref{aug_tab}).
\end{itemize}
Note that clearly, there exists $k\leq t$ such that $\locins^k(T)\in\Tab_l$.

With this definition, the insertion $j \to T$ for a tableau $T$ of shape $\al$ 
can be identified with the following procedure:
\begin{enumerate}
\item start with the augmented tableau $\tilde{T}$ of shape $(\al,b)$ such that $\tilde{T}_+=T$, $b$ is the box in the first column of $T$ with the
  smallest entry $j'$ such that $j\leq j'$, and $\tilde{T}_-(b)=j$ (this determines $T'$ by Remark \ref{aug_tab}),
\item apply $\locins$ recursively until the result is a tableau.
\end{enumerate}
In particular, the cocyclage of a tableau has the following description
in terms of $\locins$.

\begin{lemma}
\label{lemma:cyc}
Let $T$ be an authorized symplectic tableau of shape $\alpha$ and let
$r \in \mathbb{Z}_{>0}$ be such that $\localshift^{r}(\shape(T)) =
\shift(\shape(T))$. Then 

\begin{align*}
    \op{CoCyc}_{\CC}(T) = \op{red}(\locins^{r-1}(\localshift(T))). 
\end{align*}
\end{lemma}

\begin{exa}
\Yboxdim{1cm}
Take 
$T=
\begin{tikzpicture}[font=\tiny, scale=0.5,baseline=0cm]
\tgyoung(0cm,0cm,<\ov{6}>,<\ov{4}>,<\ov{3}>,<4>)
\tgyoung(1cm,-1cm,<\ov{4}>,<\ov{2}>,<6>)
\tgyoung(2cm,-2cm,<2>)
\end{tikzpicture},
$
so that
$\localshift(T)=\tilde{T}=
\begin{tikzpicture}[font=\tiny, scale=0.5,baseline=0cm]
\tgyoung(0cm,0cm,<\ov{6}>,<\ov{4}>,<\ov{3}>,<>)
\tgyoung(1cm,-1cm,<\ov{4}>,<\ov{2}>,<6>)
\draw (0,-3) -- (1,-2);
\draw (0.3,-2.3) node (a) {\hbox{\scalebox{.7}{$2$}}};
\draw (0.7,-2.7) node (b) {\hbox{\scalebox{.7}{$4$}}};
\end{tikzpicture}.
$

We have that
\begin{align*}
\Cyc_{\CC}(T)=
\locins^2(\tilde{T})
& =
\locins^2 \,
\begin{tikzpicture}[font=\tiny, scale=0.5,baseline=0cm]
\tgyoung(0cm,0cm,<\ov{6}>,<\ov{4}>,<\ov{3}>,<>)
\tgyoung(1cm,-1cm,<\ov{4}>,<\ov{2}>,<6>)
\draw (0,-3) -- (1,-2);
\draw (0.3,-2.3) node (a) {\hbox{\scalebox{.7}{$2$}}};
\draw (0.7,-2.7) node (b) {\hbox{\scalebox{.7}{$4$}}};
\end{tikzpicture}
\\
& =\locins \,
\begin{tikzpicture}[font=\tiny, scale=0.5,baseline=0cm]
\tgyoung(0cm,0cm,<\ov{6}>,<\ov{5}>,<\ov{3}>,<2>)
\tgyoung(1cm,-1cm,<\ov{4}>,<\ov{2}>,<>)
\draw (1,-3) -- (2,-2);
\draw (1.3,-2.3) node (a) {\hbox{\scalebox{.7}{$5$}}};
\draw (1.7,-2.7) node (b) {\hbox{\scalebox{.7}{$6$}}};
\end{tikzpicture}
\quad\text{ by \ref{1} of \Cref{prop:Insertion} }
\\
& =
\begin{tikzpicture}[font=\tiny, scale=0.5,baseline=0cm]
\tgyoung(0cm,0cm,<\ov{6}>,<\ov{5}>,<\ov{3}>,<2>)
\tgyoung(1cm,-1cm,<\ov{4}>,<\ov{2}>,<5>)
\tgyoung(2cm,-3cm,<6>)
\end{tikzpicture}
\quad\text{ by \ref{1} of \Cref{prop:Insertion}. }
\end{align*}
\end{exa}


\section{Insertion and shifting}
\label{insandshift}
In this section we will construct the new algorithm computing
$\Cyc^{k}_{C}(T)$ for arbitrary $k>0$ and for $T \in \SympTab((p))$, that is $T$ is 
a symplectic tableau of row shape. 
Our algorithm does not rely on the particular form of
$\Cyc^{k-1}_{C}(T)$, which allows us to overcome the problem of
controlling many
local dependencies present in Lecouvey's original algorithm.
This will enable us to prove \Cref{conji} in \Cref{onerow} for $\lambda = (p)$ and
arbitrary $\mu$.

\subsection{Main algorithm.}

For a composition $\al$ and a box $b=(i,-j)$ of $\cD_\al$, we denote
\[i=\op{col}_{\alpha}(b) \in \mathbb{Z}_{>0} \text{ \quad and \quad } j=\op{row}_{\alpha}(b) \in \mathbb{Z}_{>0}.\]

\begin{defi}
Let $\al$ be a composition and $b$ and $b'$ be two boxes of $\alpha$ such that $b < b'$ in the natural order.
The \textit{distance} between $b$ and $b'$ in $\al$ is the nonnegative integer
\[
\delta_{\alpha}(b, b') = \row_{\alpha}(b') - \row_{\alpha}(b) - \eps_\alpha(b,b'),
\text{\quad where \quad} 
\eps_\alpha(b,b')
=\begin{cases} 1 &\text{ if $\col_{\alpha}(b) \geq \col_{\alpha}(b')$},\\
0 &\text{ otherwise.}\end{cases}
\]
\end{defi}

\begin{exa}
Let $\alpha = (2,3,1)$ and let $b=(1,-1)$ be the first box of $\cD_\al$ in the natural order.
Let us compute the distance between $b$ and $b'$ for every other box $b'\in\cD_\al$.
We have 
$$
\begin{array}{l}
0 = \delta_{\alpha}((1,-1), (2,-1)) = \delta_{\alpha}((1,-1),(1,-2) )
\\
1 = \delta_{\alpha}((1,-1),(2,-2)) =  \delta_{\alpha}((1,-1),(3,-2))=\delta_{\alpha}((1,-1),(1,-3)).
\end{array}
$$
\end{exa}

\medskip

For the rest of this section, we will consider the following situation.
Let $T \in \SympTab_n((p),\mu)$ for some positive integer $p$ and some
partition $\mu=(\mu_{\ov{n}},\dots,\mu_{\ov{1}})$.
By \Cref{prop/def}, there exists $(k_1,\dots,k_n) \in
\Z_{\geq 0}^n$ such that $T$ is the unique element of the set
\[\SSYTab_{\mathcal{C}_n}((p), (k_n+\mu_{\ov{n}},k_{n-1}+\mu_{\ov{n-1}},\dots,
k_{1}+\mu_{\ov{1}},k_1,\dots,k_n)).\]
Let $\al=\wshift^k((p),\mu)_{1}$ for some integer $k \geq 0$.
Note that when computing $\wshift^k((p),\mu)$, we reduced the number of
parts in the corresponding pair of a composition and a partition
precisely $\nred := \ell(\mu) - \ell(\wshift^k((p),\mu))_2$
times. Moreover, strictly from the definition of $\wshift$ we know that
$|\al| = p - \sum_{R \leq i \leq n}\mu_{\ov{i}}$, where $R = n-\nred+1$.
It will be convenient to consider the following tableau, which keeps track of reductions.

\begin{defi}
With the previous notations,
we denote $T_\al$ the tableau of row shape obtained from $T$ by
\begin{enumerate}
\item deleting $\mu_{\overline{i}}$ occurences of $\overline{i}$ for $i=R,\ldots,n$,
\item increasing all unbarred entries (respectively decreasing all barred entries) of the resulting tableau by $\nred$.
\end{enumerate}
\end{defi}
\begin{rem}
In other words, $T_\al$ is the unique
element of the set $\SSYTab_{\mathcal{C}_{n+\nred}}((|\al|), \nu)$
with $\nu = (k_n,\dots,
k_R,k_{R-1}+\mu_{\ov{R-1}},\dots,k_{1}+\mu_{\ov{1}},\underbrace{0,...,0}_{2\nred},k_1,\dots,k_n)$. 
Also, note that for $\mu=0$, $T_\al$ coincides with $T$ for all $\al$.
\end{rem}

\begin{exa}
Let
$T=
\begin{tikzpicture}[font=\tiny, scale=0.5,baseline=0.1cm]
\tgyoung(0cm,0cm,<\ov{2}><\ov{2}><\ov{2}><\ov{1}><2>);
\end{tikzpicture},
$
so that , $p=5$, $n=2$, $\mu=(2,1)$ and $(k_2,k_1)=(1,0)$.
Take $k=4$, and compute $\wshift^{4}((p),\mu) = ((2,1),(1))$. 
We see that we have made one reduction, so that $\nred=1$, and $R=2$.
We get $T_\al = 
\begin{tikzpicture}[font=\tiny, scale=0.5,baseline=0.1cm]
\tgyoung(0cm,0cm,<\ov{3}><\ov{2}><3>);
\end{tikzpicture}
$
with the convention that boxes are labeled by
$[1,|\al|]$ in the natural order: $T_\al(1) = \ov{3}$, $T_\al(2) = \ov{2}$ and $T_\al(3) = 3$. 
\end{exa}

We are ready to describe our construction of a tableau $T_{k}$ of shape $\alpha$, which we
will later show to be equal to $\Cyc^{k}(T)$. This construction is
very algorithmic in nature, and its formal definition is given by
\Cref{alg:content}. In order to help the reader going smoothly through
this formal definition we will describe first the main idea of the
algorithm and we will illustrate it by two examples.

Notice first that for any symplectic tableau of weight $\mu$ and shape
$\lambda$ the number $|\lambda|-|\mu|$ is always even. Indeed, there
are precisely $|\lambda|-|\mu|$ boxes in this tableau, whose multiset of
contents $\op{Cont}$
is invariant by changing each content
into its opposite, that is $\op{Cont} = \overline{\op{Cont}}$. Moreover, these
contents are non-decreasing in the natural order, therefore we can
naturally match the corresponding boxes into pairs, called
\emph{partners}, such that their contents are opposite. Finally, if
we know $\mu$ and if we know the positive contents of the partners, we
can recover our initial tableau. This idea illustrates how our algorithm, consisting
in two main steps, works:
\begin{enumerate}
\item Decompose the set of boxes of $\alpha$ into two disjoint sets: \emph{partners} and \emph{singles}.
\item Assign a content to each box to obtain the tableau $T_k$. This procedure will depend on whether the box is a partner or a single.
\end{enumerate}
If $\mu = 0$, then the set of singles is empty, so the first step in our
algorithm is trivial. We start by analyzing this example, which is
much simpler and gives
a good insight of how the second step of the algorithm is working.

\begin{exa}[Weight zero]
\label{example:weightzero} 
Let $T$ be a tableau of shape $(2q)$ and weight zero (note that all tableaux of weight zero must have an even number of boxes).
We label its boxes by elements in the interval $[1,2q]$.
Fix a nonnegative integer $k$ and let $\alpha = \wshift^{k}((2q),0)_{1}  = \shift^{k}((2q))$.
Note that in this case, we always have $|\al|=2q$.
The boxes of $\alpha$ are then labeled by $[1,2q]=[1,|\al|]$ by
enumerating them in the natural order and we will write $D$ for a box $b=\square_D \in \mathcal{D}_\alpha$.
Its \textit{partner} is the box $D'=2q-D+1$.
Now, we define the tableau $T_k$ by assigning a content to the boxes of $\alpha$ as follows.
For a box $D$ of $\alpha$
\begin{equation}
  \label{eq:contentAA}
  T_{k}(D) = \begin{cases} T_{\alpha}(D) +
  \delta_{\alpha}({D'},D) &\text{ if } D >q,\\
\overline{T_{k}(D')} &\text{ if } D \leq q.
\end{cases}
\end{equation}
In particular, partner boxes have opposite contents.

\medskip

\Yboxdim{1cm}
For instance, take $q=2$ and
$T= T_{0} = 
\begin{tikzpicture}[font=\tiny, scale=0.5,baseline=0.1cm]
\Yfillcolour{cyan!30}
\tgyoung(0cm,0cm,<\ov{1}><\ov{1}><1><1>)
       \Yfillcolour{magenta!30}
\tgyoung(0cm,0cm,<\ov{1}>)
\tgyoung(3cm,0cm,<1>)
\end{tikzpicture},
$
\noindent
where we have identified partners by shading them in with the same color. 
We check that $m_0((4))=6$ (using \Cref{numsteps}), and we can compute
all the $T_k$ for $k=1,\ldots, 6$.
$$T_{1} = 
\begin{tikzpicture}[font=\tiny, scale=0.5,baseline=0cm]
\Yfillcolour{cyan!30}
\tgyoung(0cm,0cm,<\ov{1}><\ov{1}><1>,<1>)
       \Yfillcolour{magenta!30}
\tgyoung(0cm,0cm,<\ov{1}>)
\tgyoung(0cm,-1cm,<1>)
\end{tikzpicture},
T_{2} = 
\begin{tikzpicture}[font=\tiny, scale=0.5,baseline=0cm]
\Yfillcolour{cyan!30}
\tgyoung(0cm,0cm,<\ov{2}><\ov{1}>,<1><2>)
       \Yfillopacity{1}
       \Yfillcolour{magenta!30}
\tgyoung(0cm,0cm,<\ov{2}>)
\tgyoung(1cm,-1cm,<2>)
\end{tikzpicture},
T_{3} = 
\begin{tikzpicture}[font=\tiny, scale=0.5,baseline=0cm]
\Yfillcolour{cyan!30}
\tgyoung(0cm,0cm,<\ov{2}>,<\ov{1}><1><2>)
       \Yfillopacity{1}
       \Yfillcolour{magenta!30}
\tgyoung(0cm,0cm,<\ov{2}>)
\tgyoung(2cm,-1cm,<2>)
\end{tikzpicture},
T_{4} = 
\begin{tikzpicture}[font=\tiny, scale=0.5,baseline=0cm]
\Yfillcolour{cyan!30}
\tgyoung(0cm,0cm,<\ov{2}>,<\ov{1}><1>,<2>)
       \Yfillopacity{1}
       \Yfillcolour{magenta!30}
\tgyoung(0cm,0cm,<\ov{2}>)
\tgyoung(0cm,-2cm,<2>)
\end{tikzpicture},
T_{5} = 
\begin{tikzpicture}[font=\tiny, scale=0.5,baseline=0cm]
\Yfillcolour{cyan!30}
\tgyoung(0cm,0cm,<\ov{3}>,<\ov{1}>,<1><3>)
       \Yfillopacity{1}
       \Yfillcolour{magenta!30}
\tgyoung(0cm,0cm,<\ov{3}>)
\tgyoung(1cm,-2cm,<3>)
\end{tikzpicture}
\text{ and }
T_{6} = 
\begin{tikzpicture}[font=\tiny, scale=0.5,baseline=0cm]
\Yfillcolour{cyan!30}
\tgyoung(0cm,0cm,<\ov{3}>,<\ov{1}>,<1>,<3>)
       \Yfillopacity{1}
       \Yfillcolour{magenta!30}
\tgyoung(0cm,0cm,<\ov{3}>)
\tgyoung(0cm,-3cm,<3>)
\end{tikzpicture}.
$$
\end{exa}

\medskip

Now, we will explain the general case when a weight $\mu$ is
arbitrary. We already noticed that when $\mu=0$ the first part of the
algorithm, namely finding partners, is trivial. However, for arbitrary weights this part of the
algorithm is the most complex one. The procedure of finding partners
is achieved recursively and is somehow dependent on assigning
contents, that is on the second step of the algorithm. Therefore we
are performing both steps simultaneously as follows.
All boxes $S \in \mathcal{D}_{\alpha}$ such that $T_{\alpha}(S)$ is unbarred will have a partner, 
and to assign such a partner, we start with the minimal such $S$ and
we will recursively increase it. In
order to do this we introduce a variable $D = \min\{ S\in[1,|\al|]\mid
T_\al(S) \text{ is unbarred}\}$ (see \cref{alg:content}). Then, we will check the barred letters of $T_\alpha$ one by one, 
starting from $D' = \max\{ S\in[1,|\al|]\mid T_\al(S) \text{ is
barred}\}$, until we find the right partner for $D$.
To decide this we first set $M = 1$ and compute the quantity 
\begin{align}
\label{precon}
X = T_\al(D)+\de_\al(D',D).
\end{align}
Now, if 
\begin{align}
\label{ineqm1}
X<M+\nred
\text{\quad or \quad}
M\geq n-\nred+1
\end{align} 
then we declare the boxes $D$ and $D'$ to be partners and set their contents to be 
$T_{k}(D) = X$ and $T_{k}(D') = \ov{X}$. 
Then we iterate and compute the quantity (\ref{precon}) for  $D+1$ and $D'-1$, respectively, that is, 
we go on to find a partner for $D+1$ by checking first $D'-1$. 
If these conditions are not satisfied, we will declare $D'$ as well as all 
$S$ such that $S\in[D'-\mu_{\ov{M}}+1,D']$ to be single, and we define $T_{k}(S)=\ov{M+\nred}$. 
We then continue to look for a partner for $D$ by computing (\ref{precon}) 
for $D$ and $D'-\mu_{\ov{M}}$  and checking inequalities (\ref{ineqm1}) for $M := M+1$.

\begin{exa}
Let us see what this means for 
\begin{center}
\Yboxdim{1cm}
$T= T_{0} = 
\begin{tikzpicture}[font=\tiny, scale=0.5,baseline=0.1cm]
\tgyoung(0cm,0cm,<\ov{2}><\ov{2}><\ov{2}><\ov{1}><\ov{1}><\ov{1}><1><1><1>)
       \Yfillopacity{1}
       \Yfillcolour{magenta!30}
\tgyoung(5cm,0cm,<\ov{1}>)
\tgyoung(6cm,0cm,<1>)
       \Yfillopacity{1}
       \Yfillcolour{cyan!30}
\tgyoung(4cm,0cm,<\ov{1}>)
\tgyoung(7cm,0cm,<1>)
       \Yfillopacity{1}
       \Yfillcolour{gray!30}
\tgyoung(3cm,0cm,<\ov{1}>)
\tgyoung(8cm,0cm,<1>)
\end{tikzpicture}.
$
\end{center}
We have colored in partners with the same color and left singles in
white.
It is easy to check that when we construct $T_k$ for $k\leq 4$ we are assigning partners one by one,
so that the algorithm works similarly as in the weight zero
case:
$$
\begin{array}{lll}
\Yboxdim{1cm}
T_{1} = 
\begin{tikzpicture}[font=\tiny, scale=0.5,baseline=0.1cm]
\tgyoung(0cm,0cm,<\ov{2}><\ov{2}><\ov{2}><\ov{1}><\ov{1}><\ov{1}><1><1>,<1>)
       \Yfillopacity{1}
       \Yfillcolour{magenta!30}
\tgyoung(5cm,0cm,<\ov{1}>)
\tgyoung(6cm,0cm,<1>)
       \Yfillopacity{1}
       \Yfillcolour{cyan!30}
\tgyoung(4cm,0cm,<\ov{1}>)
\tgyoung(7cm,0cm,<1>)
       \Yfillopacity{1}
       \Yfillcolour{gray!30}
\tgyoung(3cm,0cm,<\ov{1}>)
\tgyoung(0cm,-1cm,<1>)
\end{tikzpicture},
\medskip
&
T_2 =
\begin{tikzpicture}[font=\tiny, scale=0.5,baseline=0.1cm]
\tgyoung(0cm,0cm,<\ov{2}><\ov{2}><\ov{2}><\ov{1}><\ov{1}><\ov{1}><1>,<1><1>)
       \Yfillopacity{1}
       \Yfillcolour{magenta!30}
\tgyoung(5cm,0cm,<\ov{1}>)
\tgyoung(6cm,0cm,<1>)
       \Yfillopacity{1}
       \Yfillcolour{cyan!30}
\tgyoung(4cm,0cm,<\ov{1}>)
\tgyoung(0cm,-1cm,<1>)
       \Yfillopacity{1}
       \Yfillcolour{gray!30}
\tgyoung(3cm,0cm,<\ov{1}>)
\tgyoung(1cm,-1cm,<1>)
\end{tikzpicture},
&
T_3 =
\begin{tikzpicture}[font=\tiny, scale=0.5,baseline=0.1cm]
\tgyoung(0cm,0cm,<\ov{2}><\ov{2}><\ov{2}><\ov{1}><\ov{1}><\ov{1}>,<1><1><1>)
       \Yfillopacity{1}
       \Yfillcolour{magenta!30}
\tgyoung(5cm,0cm,<\ov{1}>)
\tgyoung(0cm,-1cm,<1>)
       \Yfillopacity{1}
       \Yfillcolour{cyan!30}
\tgyoung(4cm,0cm,<\ov{1}>)
\tgyoung(1cm,-1cm,<1>)
       \Yfillopacity{1}
       \Yfillcolour{gray!30}
\tgyoung(3cm,0cm,<\ov{1}>)
\tgyoung(2cm,-1cm,<1>)
\end{tikzpicture},
\\
T_4 =
\begin{tikzpicture}[font=\tiny, scale=0.5,baseline=0.1cm]
\tgyoung(0cm,0cm,<\ov{2}><\ov{2}><\ov{2}><\ov{1}><\ov{1}>,<\ov{1}><1><1><1>)
       \Yfillopacity{1}
       \Yfillcolour{magenta!30}
\tgyoung(0cm,-1cm,<\ov{1}>)
\tgyoung(1cm,-1cm,<1>)
       \Yfillopacity{1}
       \Yfillcolour{cyan!30}
\tgyoung(4cm,0cm,<\ov{1}>)
\tgyoung(2cm,-1cm,<1>)
       \Yfillopacity{1}
       \Yfillcolour{gray!30}
\tgyoung(3cm,0cm,<\ov{1}>)
\tgyoung(3cm,-1cm,<1>)
\end{tikzpicture}.
&&
\end{array}
$$
However, for $k=5$ and $k=6$ the partners are reassigned by the algorithm. Note
that at this stage, if we
decided to keep partners as they were so far and to compute their
content by (\ref{eq:contentAA}), we would still get the correct tableau:
\begin{center}
$
T_5 =
\begin{tikzpicture}[font=\tiny, scale=0.5,baseline=0.1cm]
\tgyoung(0cm,0cm,<\ov{2}><\ov{2}><\ov{2}><\ov{2}>,<\ov{1}><\ov{1}><1><1><2>)
       \Yfillopacity{1}
       \Yfillcolour{magenta!30}
\tgyoung(1cm,-1cm,<\ov{1}>)
\tgyoung(2cm,-1cm,<1>)
       \Yfillopacity{1}
       \Yfillcolour{cyan!30}
\tgyoung(0cm,-1cm,<\ov{1}>)
\tgyoung(3cm,-1cm,<1>)
       \Yfillopacity{1}
       \Yfillcolour{gray!30}
\tgyoung(0cm,0cm,<\ov{2}>)
\tgyoung(4cm,-1cm,<2>)
\end{tikzpicture}
=
\begin{tikzpicture}[font=\tiny, scale=0.5,baseline=0.1cm]
\tgyoung(0cm,0cm,<\ov{2}><\ov{2}><\ov{2}><\ov{2}>,<\ov{1}><\ov{1}><1><1><2>)
       \Yfillopacity{1}
       \Yfillcolour{magenta!30}
\tgyoung(1cm,-1cm,<\ov{1}>)
\tgyoung(2cm,-1cm,<1>)
       \Yfillopacity{1}
       \Yfillcolour{cyan!30}
\tgyoung(0cm,-1cm,<\ov{1}>)
\tgyoung(3cm,-1cm,<1>)
       \Yfillopacity{1}
       \Yfillcolour{gray!30}
\tgyoung(3cm,0cm,<\ov{2}>)
\tgyoung(4cm,-1cm,<2>)
\end{tikzpicture},
\quad
T_6 =
\begin{tikzpicture}[font=\tiny, scale=0.5,baseline=0.1cm]
\tgyoung(0cm,0cm,<\ov{2}><\ov{2}><\ov{2}><\ov{2}>,<\ov{1}><\ov{1}><1><1>,<2>)
       \Yfillopacity{1}
       \Yfillcolour{magenta!30}
\tgyoung(1cm,-1cm,<\ov{1}>)
\tgyoung(2cm,-1cm,<1>)
       \Yfillopacity{1}
       \Yfillcolour{cyan!30}
\tgyoung(0cm,-1cm,<\ov{1}>)
\tgyoung(3cm,-1cm,<1>)
       \Yfillopacity{1}
       \Yfillcolour{gray!30}
\tgyoung(0cm,0cm,<\ov{2}>)
\tgyoung(0cm,-2cm,<2>)
\end{tikzpicture}
=
\begin{tikzpicture}[font=\tiny, scale=0.5,baseline=0.1cm]
\tgyoung(0cm,0cm,<\ov{2}><\ov{2}><\ov{2}><\ov{2}>,<\ov{1}><\ov{1}><1><1>,<2>)
       \Yfillopacity{1}
       \Yfillcolour{magenta!30}
\tgyoung(1cm,-1cm,<\ov{1}>)
\tgyoung(2cm,-1cm,<1>)
       \Yfillopacity{1}
       \Yfillcolour{cyan!30}
\tgyoung(0cm,-1cm,<\ov{1}>)
\tgyoung(3cm,-1cm,<1>)
       \Yfillopacity{1}
       \Yfillcolour{gray!30}
\tgyoung(3cm,0cm,<\ov{2}>)
\tgyoung(0cm,-2cm,<2>)
\end{tikzpicture}.
$
\end{center}
This coincidental correctness is broken when $k=8$ (for $k=7$ partners
are assigned the same as at the beginning). \cref{alg:content} produces
the following tableau
\begin{center}
$
T_8 =
\begin{tikzpicture}[font=\tiny, scale=0.5,baseline=0.1cm]
\tgyoung(0cm,0cm,<\ov{3}><\ov{2}><\ov{2}>,<\ov{2}><\ov{1}><\ov{1}><1>,<1><3>)
       \Yfillopacity{1}
       \Yfillcolour{magenta!30}
\tgyoung(2cm,-1cm,<\ov{1}>)
\tgyoung(3cm,-1cm,<1>)
       \Yfillopacity{1}
       \Yfillcolour{cyan!30}
\tgyoung(1cm,-1cm,<\ov{1}>)
\tgyoung(0cm,-2cm,<1>)
       \Yfillopacity{1}
       \Yfillcolour{gray!30}
\tgyoung(0cm,0cm,<\ov{3}>)
\tgyoung(1cm,-2cm,<3>)
\end{tikzpicture}
\neq
\begin{tikzpicture}[font=\tiny, scale=0.5,baseline=0.1cm]
\tgyoung(0cm,0cm,<\ov{2}><\ov{2}><\ov{2}>,<\ov{2}><\ov{1}><\ov{1}><1>,<1><2>)
       \Yfillopacity{1}
       \Yfillcolour{magenta!30}
\tgyoung(2cm,-1cm,<\ov{1}>)
\tgyoung(3cm,-1cm,<1>)
       \Yfillopacity{1}
       \Yfillcolour{cyan!30}
\tgyoung(1cm,-1cm,<\ov{1}>)
\tgyoung(0cm,-2cm,<1>)
       \Yfillopacity{1}
       \Yfillcolour{gray!30}
\tgyoung(0cm,-1cm,<\ov{2}>)
\tgyoung(1cm,-2cm,<2>)
\end{tikzpicture},$
\end{center}
where the tableau on the right hand side is obtained by keeping
partners the same as at the beginning and applying (\ref{eq:contentAA})
to compute their contents. Note that this tableau is not even semistandard.
Let us analyze in details the case $k=9$. We claim that the algorithm
produces the following tableau
\begin{center}
$
T_9 =
\begin{tikzpicture}[font=\tiny, scale=0.5,baseline=0.1cm]
\tgyoung(0cm,0cm,<\ov{3}><\ov{2}><\ov{2}>,<\ov{2}><\ov{1}><\ov{1}>,<1><1><3>)
       \Yfillopacity{1}
       \Yfillcolour{magenta!30}
\tgyoung(2cm,-1cm,<\ov{1}>)
\tgyoung(0cm,-2cm,<1>)
       \Yfillopacity{1}
       \Yfillcolour{cyan!30}
\tgyoung(1cm,-1cm,<\ov{1}>)
\tgyoung(1cm,-2cm,<1>)
       \Yfillopacity{1}
       \Yfillcolour{gray!30}
\tgyoung(0cm,0cm,<\ov{3}>)
\tgyoung(2cm,-2cm,<3>)
\end{tikzpicture}$.
\end{center}
Indeed, we first need to find a partner for $D=7$, which is colored in
pink. Since $T_\alpha(D) = 1$, before we find a partner of $D$, we assign a content to all the single boxes
which correspond to $\mu_{\ov{1}}$. 
This is what is happening in the first step of the algorithm: $M=1$, $D'=6$, and the distance
between $D$ and $D'$ in $\alpha = (3,3,3)$ is equal to $0$, so
$T_\alpha(D)+\delta_\al(D',D) = 1 \nless M$. In our case
$\mu_{\ov{1}}=0$, therefore nothing is happening except that we
increase $M$ and now since $T_\alpha(D)+\delta_\al(D',D) = 1 < M$ we
assign $D'$ to be the partner of $D$, and we update $D=8, D'=5$. These
boxes will also be partners (and they are colored in blue) since
$T_\alpha(D)=1$, and their distance is still equal to $0$. Updating
$D=9,D'=4$, we see that $T_\alpha(D) = 2$, therefore our algorithm is assigning a content to all the single boxes
which correspond to $\mu_{\ov{2}}$. Indeed,
$T_\alpha(D)+\delta_\al(D',D) = 3$, since $\delta_\al(D',D)=1$ and
this is not less than $M=2$, therefore $T_k(S) = \ov{2}$ for all $S
\in [D'-\mu_{\ov{2}}+1,D'] = [2,4]$, and we update $D'=1$ and
$M=3$. Finally, $M$ is bigger then the number of distinct positive
contents $R-1$ in $T_\alpha$, therefore $D$ and $D'$ are automatically
partners with contents $T_{\al}(D)+\delta_\al(D',D) = 3$ and $\ov{3}$, respectively.
\end{exa}

\begin{algorithm}[h]
\caption{Defining the tableau $T_{k}$.}
\label{alg:content}
\begin{algorithmic} 
\Require Nonnegative integers $k,k_1,\dots,k_n$ and a partition $\mu=(\mu_{\ov{n}},\dots,\mu_{\ov{1}})$.
\Ensure The tableau $T_{k}: \cD_\al\ra \cC$ of shape $\al$.
\State $p = \sum_{i = 1}^n(2k_i+\mu_{\ov{i}})$
\State $\al = \wshift^{k}((p),\mu)_{1}$
\State $\nred = \ell(\mu)-\ell(\wshift^{k}((p),\mu)_{2})$
\State $R = n-\nred +1$
\State $D = \min\{ S\in[1,|\al|]\mid T_\al(S) \text{ is unbarred}\}$
\State $D' = \max\{ S\in[1,|\al|]\mid T_\al(S) \text{ is barred}\}$
\State $M = 1$
\While{$D\leq |\al|$}
\State $\partners = \false$
    \State $X = T_\al(D)+\de_\al(D',D)$
    \If{$\partners == \false$} 
        \If{$X<M+\nred$ or $M\geq R$} 
            \State $\partners = \true$
            \State $T_{k}(D') = \ov{X}$, $T_{k}(D)=X$
            (the boxes $D'$ and $D$ are said to be \textit{partners})
            \State $D=D+1$, $D'=D'-1$
        \Else
            \State $T_{k}(S)=\ov{M+\nred}$ for all $S\in[D'-\mu_{\ov{M}}+1,D']$
            \State $D'=D'-\mu_{\ov{M}}$
            \State $M=M+1$
        \EndIf
    \EndIf
\EndWhile
\If{$D'>1$}
    \State{$T_k(S)=T_\al(S)$ for all $S\in[1,D'-1]$}
\EndIf
\end{algorithmic}
\end{algorithm}

\begin{rem}\label{rem_philo_2}
We see that the tableau $T_k$ is determined by: 
\begin{itemize}
\item The composition shape $\al$ obtained by shifting $k$ times. This data is inherited from type $\AAA$, as explained in \Cref{sec:CyclageA}.
\item The tableau $T$ (which determines $T_\al$), which can be understood as the type $\CC$ ``initial data''.
\end{itemize}
\end{rem}

\subsection{Local shifting}
\label{locshift}

In order to prove \Cref{thm:crucial}, we need to refine the construction of $T_k$
into tableaux of augmented shapes, that will be denoted $T_{k,s}$.
From now on, we will systematically identify a box with its label
(obtained from the natural order).

Take $p,\mu$ and $k$ as before and let $\al =
\op{simp}\bigg(\wshift^k\big((p),\mu\big)\bigg)_1$. Let $r = \alpha_{j+1}$,
where $j = \min\{i: \alpha_i = \max_k\alpha_k\}$ and for any $1 \leq s
\leq r$, set $\al^s =\localshift^s(\alpha)$, so that $\al^s$ is an augmented composition. 
Let $c,c+1 \in [1,|\al|]$ denote the labels (in the natural order) of the augmented boxes in $\alpha^s$. 
We define $\pos_{\al,s}:[1,|\al|]\to [1,|\al|]$ as
\[
  \pos_{\al,s}(x) = \begin{cases}
    x+1 &\text{ if } x \in [c+1-s,c),\\
    x &\text{ otherwise. }
   \end{cases}\]
and we set
\[
  \delta_{\alpha^s}(x,y) = \begin{cases}
    \delta_{\alpha}(\pos_{\al,s} (x), \pos_{\al,s} (y)) &\text{ if }
    x,y \neq c\text{ or } s=1\\
    \delta_{\alpha}(\pos_{\al,s} (x),c) &\text{ if }
    s>1,y=c,T_{\al}(c)\text{ is barred, }\\
            \delta_{\alpha}(\pos_{\al,s} (x),c+1) &\text{ if }
            s>1,y=c,T_{\al}(c)\text{ is not barred, }\\
                \delta_{\alpha}(c, \pos_{\al,s} (y)) &\text{ if }
    s>1,x=c,T_{\al}(c)\text{ is barred, }\\
            \delta_{\alpha}(c+1, \pos_{\al,s} (y)) &\text{ if }
    s>1,x=c,T_{\al}(c)\text{ is not barred. }
   \end{cases}\]
Finally, we define a tableau $T_{k,s}$ of shape
$\alpha^s$ by applying the following modification of \Cref{alg:content}: instead of $\alpha,\delta_\alpha,\nred$ we
use $\alpha^s, \delta_{\alpha^s}$ and $\nred'=\ell(\mu)-\ell\bigg(\wshift^{k+1}\big((p),\mu\big)_2\bigg)$ respectively.

\medskip

The tableaux $T_{k}$ (respectively $T_{k,s}$) have some very useful properties, the most important of which we encompass in the following crucial lemma. 
For $x\in\cC$, denote $\I_x =T_{k}^{-1}(\{x\})$ (respectively $\I_x = T_{k,s}^{-1}(\{x\})$)
and $\I_{\leq x} = T_k^{-1}(\{y\in\cC\mid y\leq x\})$ (respectively  $\I_{\leq x} = T_{k,s}^{-1}(\{y\in\cC\mid y\leq x\})$).

\begin{lemma}
  \label{lem:crucial}
  The following properties hold true.
\begin{enumerate}
\item
\label{oaoa2}
$T_{k}$ and $T_{k,s}$ are natural tableaux, that is for all $1\leq  t < u\leq |\al|$ one has
$T_{k}(t) \leq T_{k}(u)$ and
$T_{k,s}(t) \leq T_{k,s}(u)$ (see \cref{def:NaturalTableau}).
\item
\label{oaoa'}
For all $1 \leq i \leq n$ and for all $0 \leq j < |\I_i|$we have that $\max\I_i-j = \partner(\min \I_{\ov{i}}+j)$.
\item
\label{oaoa}
For all $1 \leq i \leq n$, the functions $\delta_{\alpha}, \delta_{\alpha^{s}}$ are constant on
the product of intervals $\I_{\overline i} \times\I_i$.
\end{enumerate}
\end{lemma}

\begin{proof}

We will prove the statements for $T_{k}$, since the arguments for
$T_{k,s}$ are identical. Let $1\leq  t < u\leq |\al|$. If either $t$
or $u$ is a single then \Cref{alg:content} gives directly the desired inequality $T_{k}(t) \leq T_{k}(u)$.
Assume that $1\leq t < u\leq |\al|$ are such that
$T_\al(t)$ and $T_\al(u)$ are unbarred and let $t' = \partner(t), u' =
\partner(u)$. Note that $\delta_\al$ is bi-increasing, that is for
every $1\leq x < y < z\leq |\al|$ we have $\delta_\al(x,y) \leq
\delta_\al(x,z)$ and $\delta_\al(y,z) \leq
\delta_\al(x,z)$.
Therefore

\[ T_{k}({t}) = T_\al(t) +
  \delta_{\alpha}({t'}, {t}) \leq T_\al(u) +
  \delta_{\alpha}({u'}, {u}) = T_{k}({u}) \]
since the function 
$T_\al$ is increasing by definition.
This finishes the proof  of (\ref{oaoa2}) since for any $1 \leq d
\leq |\al|$ which is not a single we have

$$ T_{k}({\partner(d)})  = \overline{T_{k}(D)}.$$

Fix $i\in\cC$.  For $\ell\in\{i,\ov{i}\}$, let ${\ell}^{\min} = \min \I_{\ell}$ 
and $\ell^{\max} = \max \I_\ell$. 
By monotonicity of $\delta_{\alpha}$, \eqref{oaoa} is equivalent to the following
statement:
\[ \delta_{ \alpha}({\overline i}^{\min}, i^{\max}) =
\delta_{\alpha}({\overline i}^{\max}, i^{\min}).\]
Notice first that $i^{\max}=\partner(\ov{i}^{\min})$, and more
generally \eqref{oaoa'} holds true, which is simply a
reformulation of the {\bf if} part of \Cref{alg:content} for a fixed
value of $X=i$.
Therefore, it follows from \Cref{alg:content} that 

\[i - T_\al(i^{\max}) = \delta_{\alpha}({\overline i}^{\min},
i^{\max})
 \geq \delta_{\alpha}({\overline i}^{\max},
{i}^{\min})
 \geq i - T_\al(i^{\min}) \]
and by \eqref{oaoa2} all the inequalities above are equalities. This
finishes the proof of \eqref{oaoa}.  
\end{proof}

\begin{corollary}
\label{maciekscor}
Let $i\in\cC_{\geq 0}$, $\ell\in\Z_{\geq0}$ 
and $s_{1} \leq s_{2} \leq s_{3} \leq s_{4}\in[1,|\al|]$ such that  
\begin{align*}
\overline{T_{k}({s_1})}= \overline{T_{k}({s_2}) +\ell} = T_{k}({s_3}) +\ell = T_{k}({s_4})=&i +\ell,\\
\overline{T_{k,s}({s_1})}= \overline{T_{k,s}({s_2}) +\ell} =
  T_{k,s}({s_3}) +\ell = T_{k,s}({s_4})=&i+\ell, \text{ respectively.}
\end{align*} 
\noindent
Then 
\begin{align*}
\delta_{\alpha}({s_1}, {s_4}) - \delta_{\alpha}({s_2}, {s_3}) \leq &\ell,\\
\delta_{\alpha^{s}}({s_1}, {s_4}) -
  \delta_{\alpha^{s}}({s_2}, {s_3}) \leq &\ell, \text{ respectively.}
\end{align*}

\end{corollary}

\begin{proof}
\Cref{lem:crucial} (\ref{oaoa}) implies that
\[\delta_{\alpha}({s_1}, {s_4}) - \delta_{\alpha}({s_2}, {s_3}) =
  \delta_{\alpha^s}({s_1}, {s_4}) - \delta_{\alpha^s}({s_2}, {s_3}) = 
                                                                \ell- (T_\al(\I_{i+\ell}) -T_\al(\I_{i}) )\leq \ell,\]
since $T_\al$ is increasing.
\end{proof}

\subsection{Insertion and shifting.}

In this section, we state
\Cref{thm:crucial},
which is crucial for the proof of \Cref{mainresultintro}.The proof of this result being quite technical, we delay it to \Cref{appendix}

\begin{lemma}
\label{lem:reduction}
Let $\mu = (\mu_{\ov{n}},\mu_{\ov{n-1}},\dots,\mu_{\ov{1}})$ be a
partition, $k,k_1,\dots,k_n \in \Z_{\geq 0}$, $p =
\sum_i(2k_i+\mu_i)$ and let $\alpha = \wshift^{k}(\mu, (p))$. 
Then
\begin{align}
\label{statement:reduction}
    T_{k,1} = \localshift\red(T_{k}).
\end{align}
\end{lemma}

\begin{proof}
First, note that $\shape(\localshift\red(T_k)) = \shape(T_{k,1})$, which is a direct
consequence of  \Cref{rem:reduction}. Let $\alpha = \wshift^k\big((p),\mu\big)_1$. In order to finish
the proof it is enough to show that
performing \Cref{alg:content} with $\op{simp}(\alpha,\mu)_1,
\op{simp}(\alpha,\mu)_2,\nred'=\ell(\mu)-\ell(\op{simp}(\alpha,\mu)_2)$ in
place of $\alpha,\mu,\nred$ gives us a tableau $T'$ which is equal to $\red(T_k)$. If $\red T_k = T_k$,
there is nothing to prove. Otherwise $T_k \in
\SympTab_n(\beta,\nu)$, where $\nu =
(\mu_{\ov{n-\nred}},\dots,\mu_{\ov 1},\underbrace{0,\dots,0}_{\nred})$
and $\mu_{\ov{n-\nred}} \geq \cdots \geq  \mu_{\ov{n-\nred'+1}} > 0$. Strictly from the definition of
reduction we know that $\I_{\geq n-(\nred'-\nred)} \cap \I_{\leq
  n} = \emptyset$ thus
\[\I_{>0} = \big(\I_{>0}\cap \I_{<n-(\nred'-\nred)}\big) \cup \I_{>n}.\]
In particular for any $\square \in \I_{>0}\cap
\I_{<n-(\nred'-\nred)}$ we have
\[\delta_{\op{simp}(\alpha,\mu)_1}(\partner(\square),\square) =
  \delta_{\alpha}(\partner(\square),\square),\]
but for any $\square \in \I_{>n}$ we have
\[\delta_{\op{simp}(\alpha,\mu)_1}(\partner(\square),\square) =
  \delta_{\alpha}(\partner(\square),\square-(\nred'-\nred)),\]
since labeling in $\op{simp}(\alpha,\mu)_1$ corresponds to removing
boxes in $T_k$ with contents
$\{\ov{n-(\nred'-\nred)+1}^{\mu_{\ov{n-\nred'+1}}},\dots,\ov{n}^{\mu_{\ov{n-\nred}}}\}$. Note
that with this identification we do not label the boxes of
$\op{simp}(\alpha,\mu)_1$ by $[1,|\op{simp}(\alpha,\mu)_1|]$, but by
\[\big[1,\partner(\square)-(\mu_{\ov{n-\nred}} + \cdots +
\mu_{\ov{n-\nred'+1}})\big]\cup \big[\sum_{i}k_i+\sum_{j \leq n-\nred}\mu_{\ov
  j},|\alpha|\big],\]
where $\square = \max \I_{n-(\nred'-\nred)}$.
Therefore, for any $\square \in \I_{>0}\cap
\I_{<n-(\nred'-\nred)}$ we have
\begin{align*}
  T(\square) &=
  \delta_{\op{simp}(\alpha,\mu)_1}\big(\partner(\square),\square \big) +
               T_{\op{simp}(\alpha,\mu)_1}(\square) \\
  &=
  \delta_{\alpha}\big (\partner(\square),\square \big)+T_{\alpha}(\square)+(\nred'-\nred)
  = T_k(\square)+(\nred'-\nred)
  \end{align*}
and for any $\square \in \I_{>n}$ we have
\begin{align*}
  T(\square) &=\delta_{\op{simp}(\alpha,\mu)_1}\big (\partner(\square),\square \big) +
    T_{\op{simp}(\alpha,\mu)_1}(\square) \\
  &=
  \delta_{\alpha}\big (\partner(\square),\square\big)+T_{\alpha}(\square)
    = T_k(\square).
  \end{align*}    
Thus indeed $T' = \red(T_k)$, and we conclude the proof.
\end{proof}

We are ready to present our main theorem.

\begin{theorem}
\label{thm:crucial}
Let $\mu = (\mu_{\ov{n}},\mu_{\ov{n-1}},\dots,\mu_{\ov{1}})$ be a
partition, $k,k_1,\dots,k_n \in \Z_{\geq 0}$ and $p =
\sum_i(2k_i+\mu_i)$. Let $\alpha = \wshift^{k}(\mu, (p))$ and let $r \in \mathbb{Z}_{>0}$ be such that $\localshift^{r}(\alpha) = \shift(\alpha)$. Then, for each $1 \leq s < r$ we have

\begin{align}
\label{statement:crucial}
\op{locins}^{s}(T_{k,1}) = T_{k,s+1}.
\end{align}
\end{theorem}

The proof of \cref{thm:crucial} is technically quite involved. Therefore, in order to
motivate the reader, we will first present all the consequences of
this result, especially \Cref{mainresultintro}, and we present the
proof of \cref{thm:crucial} in a separate \cref{appendix}

\begin{cor}
\label{cor:crucial}
Let $n,p \geq 0$ be integers and $\mu =
(\mu_{\ov{n}},\mu_{\ov{n-1}},\dots,\mu_{\ov{1}})$ a partition. For any
$T \in \op{SympTab}_{n}((p),\mu)$ we have

\begin{align}
\label{crucialfin}
 \Cyc^{k}_{\CC}(T) = \red\big(T_{k}\big).
 \end{align}
 
\end{cor}

\begin{proof}
We proceed by induction on $k$. \Cref{authorized} implies that $T$ is
authorized unless $\mu_{\ov{n}} = p$, that is, unless $\mu = (p)$. If
this is the case, then $\Cyc_{\CC}(T) = \red(T) = \emptyset$. From the
other hand, applying \Cref{alg:content} we first compute $\al =
\wshift((p),\mu)_{1}=\emptyset$, therefore $T_1 = \emptyset =
\Cyc_{\CC}(T)$, as desired. If $T$ is authorized, then
\Cref{lemma:cyc} implies that
\[    \op{CoCyc}_{\CC}(T) = \red(\localshift(T)) =
  \red(\shift(T))=\red(T_1),\]
where the last equalities comes from the fact that the shape of $T$ is
simply one row and the last entry of $T$ is strictly bigger then the
first one.
We assume now that  $\Cyc^{k}_{\CC}(T) =
\red\big(T_{k}\big)$. Therefore
\[ \Cyc^{k+1}_{\CC}(T) = \Cyc_{\CC}\bigg(\red\big(T_{k}\big)\bigg)= \red\bigg(\op{locins}^{r-1}\bigg(\localshift \bigg(\red\big(T_{k}\big)\bigg)\bigg)\bigg)\] 
by \Cref{lemma:cyc}, where $r \in \mathbb{Z}_{>0}$ is such that $\localshift^{r}\big(\shape\big(\red(T_k)\big)\big) = \shift\big(\shape\big(\red(T_k)\big)\big)$.
\noindent
Applying \Cref{thm:crucial} and \Cref{lem:reduction} to the right hand
side of the above equalities we have that
\[ \Cyc^{k+1}_{\CC}(T) = \red\big(T_{k,r}\big)\]
which, by the definition and our choice of $r$, is equal to $\red\big(T_{k+1}\big)$. This
finishes the proof.
\end{proof}

\subsection{Lecouvey's conjecture}
\label{onerow}
 
In this section we are going to apply \Cref{crucialfin} to prove
\Cref{conji} in the case of a one-row $\lambda = (p)$. We need a following proposition
due to Lecouvey, which is an easy consequence of the Morris recurrence formula described in \cite{Lecouvey2005}: 

\begin{proposition}{\cite [Proposition 3.2.3.]{Lecouvey2005}}
 \label{onerowlecouvey}
Let $\mu = (\mu_{\ov{n}},\mu_{\ov{n-1}},\dots,\mu_{\ov{1}})$ be a
partition and $p \geq |\mu|$ be a positive integer. Then 
\[ K^{\CC_n}_{(p),\mu}(q) = q^{T_n(\mu)}\cdot\sum_{T \in \SympTab_n((p),\mu)} q^{\theta_n(T)}\]
where $T_{n}(\mu) = \sum_{i=1}^{n}(n-i)\mu_{\ov i}$ and 
\[\theta_{n}(T) = \sum_{i=1}^{n}(2(n-i)+1)k_i,\]
where $T \in \SSYTab_{\mathcal{C}_n}((p), (k_n+\mu_{\ov{n}},k_{n-1}+\mu_{\ov{n-1}},\dots,
  k_{1}+\mu_{\ov{1}},k_1,\dots,k_n))$.
\end{proposition}

We are ready to prove \Cref{mainresultintro}.

  \begin{proof}[Proof of \Cref{mainresultintro}]
Let $T \in \SympTab_n((p),\mu)$, where $\mu=(\mu_{\ov n},\dots,\mu_{\ov 1})$. By \Cref{prop/def} there exists
   unique nonnegative integers $k_1,\dots,k_n$ such that $T \in \SSYTab_{\mathcal{C}_n}(\lambda, (k_n+\mu_{\ov{n}},k_{n-1}+\mu_{\ov{n-1}},\dots,
  k_{1}+\mu_{\ov{1}},k_1,\dots,k_n))$. \Cref{cor:crucial} implies that
  $m(T) = \min\{k: T_k = T_{k+1}\}$, which is simply equal to
  $m_{\mu}((p))$ defined by \eqref{eq:minimal}. \Cref{numsteps} gives us
 \begin{align*} 
m(T) &= \sum_i(n-i)\mu_{\ov i} + \frac{(p-|\mu|)(p-|\mu| + 2\ell(\mu)-1)}{2} \\
&= \sum_i(n-i)\mu_{\ov i} + (\sum_ik_i)(2\sum_ik_i + 2\ell(\mu)-1).
 \end{align*}
Let us compute $E_{C(T)}$. Notice that $C(T)$ is a
column of weight $0$ and length $\sum_ik_i$. Therefore, for any
$\square,\square+1 \in \I_{>0}$ we have
\begin{align*} C(T)(\square+1)-C(T)(\square) &=
                                               \delta_{\shape(C(T))}\big(\partner(\square+1),\square+1\big)-\\
  &-\delta_{\shape(C(T))}\big(\partner(\square),\square\big)=2.
  \end{align*}
Therefore $E_{C(T)}$ consists of all positive entries of $C(T)$ and
due to the construction given by \Cref{alg:content} we know that
$\nred = \ell(\mu)$, thus
\[ E_{C(T)} = \{i+\ell(\mu)+2j\colon 1\leq i \leq n, \sum_{l \leq i-1}k_l \leq j < \sum_{l \leq i}k_l\}.\]
 Finally
 \begin{multline*} 
\ch_n(T) = m(T) + 2\sum_{i \in E_{C(T)}}(n-i) = \\
\Big[\sum_i(n-i)\mu_{\ov i} + (\sum_ik_i)(2\sum_ik_i + 2\ell(\mu)-1)\Big] 
+ 2\Big[\sum_{1\leq i \leq n}\sum_{\sum_{l \leq i-1}k_l \leq j < \sum_{l \leq i}k_l}(n-(i+2j+\ell(\mu)))\Big]\\
= \Big[\sum_i(n-i)\mu_{\ov i} + (\sum_ik_i)(2\sum_ik_i + 2\ell(\mu)-1)\Big]
+ 2\Big[\sum_{i}(n-i)k_i - \big (\sum_ik_i \big) \big (\sum_ik_i+\ell(\mu)-1 \big)\Big]\\
= \sum_{i}(n-i)(2k_i+\mu_{\ov i}) + \sum_ik_i = T_n(\mu)+\theta_n(T)
 \end{multline*}
and comparing this with \Cref{onerowlecouvey} finishes the proof.
\end{proof}

\begin{rem}\label{rem_philo_3}
Combining \Cref{rem_philo_2} and \Cref{cor:crucial}, we see that
the type $\CC$ cocyclage is controlled by the type $\AAA$ cocyclage.
Observations suggest that this phenomenon holds in a more general setting, 
and it would be interesting to investigate this further.
\end{rem}


\section{Proof of \Cref{thm:crucial}}
\label{appendix}

Our proof is by induction on $1 \leq s < r$. Before we start we need
to introduce some notation. Let $\square, \square+1$ denote the labels of the augmented boxes of $\alpha^s$,
and let $e=T_{k,s}(\square)$ and $f=T_{k,s}(\square+1)\geq e$. 
Therefore, the augmented boxes of $\alpha^{s+1}$ are labeled by
$\square+1, \square+2$. Let $C_m$ denote the $m$-th column of $T_{k,s}$. For an entry $x$ lying in the column $C$ we
denote by $C(x) \in [1,|\alpha|]\setminus\{\square\}$ the corresponding label, that is
$x \in C$ and $T_{k,s}(C(x)) = x$. We will
proceed by going through the cases described in
\Cref{prop:Insertion}. The entry $e=T_{k,s}(\square)$ will play the role of
the entry $*$ and from now on we set $C = C_S$ which is the
column containing the augmented boxes labeled by $\square,\square+1$.

\vspace{10pt}

\noindent
\textbf{\ref{1}}.
We know that $e = i$ for some $i \in \CCC_{>0}$. First, notice that $\square+1
= C(y)$, which is a direct consequence of \Cref{lem:crucial}
\eqref{oaoa2}. Indeed, we have that $x<T_{k,s}(\square) \leq
y$, therefore the only possibilities for the position of an augmented
box is either in $C(y)$ or in the box strictly below $C(y)$
necessarily with $y=i$. However, in the latter case we have
that
\[\delta_{\alpha^s}\big(\partner(\square),\square \big)-\delta_{\alpha^s}\big (\partner(C(y)),C(y) \big)
  >0,\]
which
gives a contradiction with \Cref{maciekscor} because $C(y),\square \in \I_i$. Since $\square+1
= C(y)$, which means that $T_{k,s}(\square+1) =
y$, \Cref{maciekscor} implies that $C(y) = \max \I_y$ and $C(\ov y) = \min
\I_{\ov y}$. This is a consequence of the fact that
\[ \delta_{\al^s}\big(C (\overline{y} ),b'\big) > \delta_{\al^s}\big(C
  (\overline{y} ),C(y)\big)\]
for every $b' > C(y)$ and similarly
\[\delta_{\al^s}\big(b',C(y)\big) >\delta_{\al^s}\big(C (\overline{y}
  ),C(y)\big)\]
for every $b' < C(\ov y)$. Moreover, $C(y) = \partner(C(\ov y))$ by
\Cref{lem:crucial} \eqref{oaoa'}. We also note that for every $0< j
<b$ one has
$\delta_{\al^s}\big(C(\ov{y+j}),C(y)+1\big)-\delta_{\al^s}\big(C(\ov{y}),C(y)\big)
> j$
thus $T_{k,s}(C(y)+1) \geq y+b$ by \Cref{maciekscor}. In particular
all the boxes in the interval
$\big[C(\ov{y+b-1}) -\mu_{\ov{y+b-\nred}},C(\ov{y})\big)$ are singles filled
by $\{\ov{y+1}^{\mu_{\ov{y+1-\nred}}},\dots, \ov{y+b}^{\mu_{\ov{y+b-\nred}}}\}$.
Since $i$ is unbarred, and $C(y) = \square+1$ we have that 

\[ \delta_{\alpha^{s+1}}(\square', \square+1)
= \delta_{\alpha^{s+1}}(\square',\square+1)\]
for  $\square' \in [C(\ov{y+b-1}) -\mu_{\ov{y+b-\nred}},C(\ov
y)]\setminus C$ and
\[\delta_{\alpha^{s+1}}(\square', \square+1)
= \delta_{\alpha^{s+1}}(\square',\square+1)+1\]
for $\square' \in [C(\ov{y+b-1}) -\mu_{\ov{y+b-\nred}},C(\ov y)]\cap C$.

\noindent
This implies that performing \Cref{alg:content} to obtain $T_{k,s+1}$
gives us the same result as in $T_{k,s}$ until $D = C(y)=\square+1$. At this
moment $D'=C(\ov y), M+\nred = y+1$, so we have $X = y+1 \nless
M+\nred$ and we notice that the interval
$\big(C(\ov{y+b-1})-\mu_{\ov{y+b-\nred}},C(\ov{y})\big]$ in $T_{k,s+1}$
consists of single boxes filled by $\{\ov{y+1}^{\mu_{\ov{y+1-\nred}}},\dots, \ov{y+b}^{\mu_{\ov{y+b-\nred}}}\}$. After performing these steps we
have that $D'=C(\ov{y+b-1})-\mu_{\ov{y+b-\nred}},M+\nred=y+b+1$. Since $D'
< C(\ov{z})$ we have that $X =
\delta_{\alpha^{s+1}}(D', D)+T_{\al}(D) = y+b<M+\nred$ and
$T_{k,s+1}(\square+1) = y+b, T_{k,s+1}(C(\ov{y+b-1}))= \ov{y+b}$. At this
step of the algorithm  $D =
\square+2,D' = C(\ov{y+b-1})-\mu_{\ov{y+b-\nred}}-1$ and $M+\nred=y+b+1$, therefore we have the same parameters
of \Cref{alg:content} as at a certain point of \Cref{alg:content}
performed to construct $T_{k,s}$. Thus, all the other contents of
$T_{k,s+1}$ are the same as in $T_{k,s}$. Comparing the resulting
$T_{k,s+1}$ with \ref{1} of \Cref{prop:Insertion} we conclude the
proof in this case.

\vspace{10pt}

\textbf{\ref{2.1}}. We know that $e = \ov i$ for some $i \in
\CCC_{>0}$. \Cref{lem:crucial} \eqref{oaoa2} implies that
$\square +1 = C(i-b+1)$. Since
$T_{k,s}(\square) = \ov i$ and $T_{k,s}(\square+1) = i-b+1$ we have
by \Cref{lem:crucial} \eqref{oaoa2}
that $\square \leq \partner(C(i-b+1))<\square+1$, which is possible
only when $b=1$. 
Note that performing \Cref{alg:content} to obtain $T_{k,s+1}$ 
corresponds precisely to performing \Cref{alg:content} to obtain
$T_{k,s}$. Indeed, in both cases we start from $D = \square+1,D'=\square$ and
\[\delta_{\alpha^{s}}(\I_{\ov{i}}\times \I_{i})  =
  \delta_{\alpha^{s+1}}(\I_{\ov{i}}\times \I_{i})=0.\]
Therefore $T_{k,s+1}(x) = T_{k,s}(x)$ for all $x \in [1,|\al|]$, thus
$T_{k,s+1}$ coincides with $\locins(T_{k,s})$, which is obtained form
$T_{k,s}$ by shifting the augmented box as shown in \ref{2.1} of \Cref{prop:Insertion}. This observation concludes the
proof in this case.

\vspace{10pt}

\textbf{\ref{2.2}} We know that $e = \ov i$ for some $i \in
\CCC_{>0}$. First note that necessarily $b=1$. Otherwise 
\[ \delta_{\alpha^s}\big(C(\ov{i-b+2}), C(i-b+2)\big) -
  \delta_{\alpha^s}\big(C(\ov{i-b+1}), C(i-b+1)\big))>1,\]
which is a contradiction with \Cref{maciekscor}.
Therefore \Cref{lem:crucial} \eqref{oaoa2} implies that either $\square +1 = C(\ov{i})$
or $\square +1 = C(i)$. Assuming that $\square +1 = C(\ov{i})$ we have that both $\square,\square +1 \in \I_{\ov
  i}$ but
\begin{align}
\label{c}
    \delta_{\alpha^s}(C(i), \square) = \delta_{\alpha^s}(C(i), \square+1) + 1,
\end{align}
which contradicts \Cref{maciekscor}.
Therefore $T_{k,s}(\square) = \ov{i}, T_{k,s}(\square+1) = i$. We also note that for every $0\leq x
\leq -c$ one has
$\delta_{\al^s}\big(C(\ov{i+x}),C(i)+1\big)-\delta_{\al^s}\big(C(\ov{i}),C(i)\big)
> x$
thus $T_{k,s}(C(i)+1) = T_{k,s}(\square+2) > i-c$ by \Cref{maciekscor}.  In particular
all the boxes in the interval
$\big[C(\ov{i-c}) -\mu_{\ov{i-c+1-\nred}},C(\ov{i})\big)$ are singles filled
by $\{\ov{i+1}^{\mu_{\ov{i+1-\nred}}},\dots, \ov{i-c+1}^{\mu_{\ov{i-c+1-\nred}}}\}$.
Since $i$ is unbarred, and $C(i) = \square+1$ we have that 

\[ \delta_{\alpha^{s+1}}(\square', \square+1)
= \delta_{\alpha^{s+1}}(\square',\square+1)\]
for  $\square' \in \big[C(\ov{i-c}) -\mu_{\ov{i-c+1-\nred}},C(\ov{i})\big]\setminus C$ and
\[\delta_{\alpha^{s+1}}(\square', \square+1)
= \delta_{\alpha^{s+1}}(\square',\square+1)+1\]
for $\square' \in \big[C(\ov{i-c}) -\mu_{\ov{i-c+1-\nred}},C(\ov{i})\big]\cap C$.

\noindent
This implies that performing \Cref{alg:content} to obtain $T_{k,s+1}$
gives us the same result as in $T_{k,s}$ until $D = C(i)=\square+1$. At this
moment $D'=C(\ov i), M+\nred = i+1$, so we have $X = i+1 \nless
M+\nred$ and we notice that the interval
$\big(C(\ov{i-c})-\mu_{\ov{i-c+1-\nred}},C(\ov{i})\big]$ in $T_{k,s+1}$
consists of single boxes filled by $\{\ov{i+1}^{\mu_{\ov{i+1-\nred}}},\dots, \ov{i-c+1}^{\mu_{\ov{i-c+1-\nred}}}\}$. After performing these steps we
have that $D'=C(\ov{i-c})-\mu_{\ov{i-c+1-\nred}},M+\nred=i-c+2$. Since $D'
< C(\ov{m})$ we have that $X =
\delta_{\alpha^{s+1}}(D', D)+T_{\al}(D) = i-c+1<M+\nred$ therefore
$T_{k,s+1}(\square+1) = i-c+1, T_{k,s+1}(C(\ov{i-c}))= \ov{i-c+1}$. At this
step of the algorithm  $D =
\square+2,D' = C(\ov{i-c})-\mu_{\ov{i-c+1-\nred}}-1$ and $M+\nred=i-c+2$, therefore we have the same parameters
of \Cref{alg:content} as at a certain point of \Cref{alg:content}
performed to construct $T_{k,s}$. Thus, all the other contents of
$T_{k,s+1}$ are the same as in $T_{k,s}$. Comparing the resulting
$T_{k,s+1}$ with \ref{2.2} of \Cref{prop:Insertion} we conclude the
proof in this case.

\vspace{10pt}

\textbf{\ref{2.3}}. We know that $e = \ov i$ for some $i \in
\CCC_{>0}$. We will show that in this case we necessarily have
$b=1$. Suppose that $b>1$ and notice that necessarily $y \leq \ov{i}$. Otherwise
$\partner(C(i))<y$, and $\partner(i-1)>y$ thus
\[ \delta_{\alpha^s}\big(\partner(C(i)), C(i)\big) -
  \delta_{\alpha^s}\big(\partner(C(i-1)), C(i-1)\big)>1,\]
which contradicts \Cref{maciekscor}. Therefore \Cref{lem:crucial} \eqref{oaoa2} implies that
either $\square +1 = C(y)$ (which can happen only if $y = \ov{i}$) or
$\square +1 = C(x)$. If $\square +1 = C(y)=C(\ov i)$ then both $\square,\square +1
\in \I_{\ov i}$ but
\[ \delta_{\alpha^s}\big(\square), C(i)\big) =
  \delta_{\alpha^s}\big(\square+1, C(i)\big)+1,\]
which is impossible by \Cref{maciekscor}. Therefore $\square+1 = C(x)$
so $T_{k,s}(\square) = \ov i$ and $T_{k,s}(\square+1) = x \geq \ov{i-b+1}$. \Cref{lem:crucial} \eqref{oaoa2}
implies that $\partner(C(i))\leq \square$ and $\partner(C(i-1))>\square$,
thus
\[ \delta_{\alpha^s}\big(\partner(C(i)), C(i)\big) -
  \delta_{\alpha^s}\big(\partner(C(i-1)), C(i-1)\big)>1,\]
which contradicts \Cref{maciekscor}. This finishes the proof of our
claim that $b=1$. In particular \Cref{maciekscor} implies that 
\begin{align}
  \label{eq:pomoc}
  x >
  \ov i\text{ and }T_{k,s}(C(i)-1)<i
\end{align}
since
\[ \delta_{\alpha^s}\big(C(x), C(i)\big) =
  \delta_{\alpha^s}\big(\square, C(i)-1\big) =
    \delta_{\alpha^s}\big(\square, C(i)\big)-1,\]
and $\square \in \I_{\ov i}$.

\noindent
It clear from the definition of \Cref{alg:content} that until $D'>
\square +1$ the steps of constructing $T_{k,s}$ and
$T_{k,s+1}$ coincide. In particular when $D = C(i)-1$ we know by
\eqref{eq:pomoc} that $D'=\square +1$ and since $T_{k,s}(\square) =
\ov i < T_{k,s}(D')$ \Cref{lem:crucial} \eqref{oaoa'} implies
that $\partner(C(i)-1) = \square+1$ and
\[\ov x = T_{k,s}(C(i)-1) = \delta_{\alpha^s}(\square+1,
  C(i)-1)+T_\al(C(i)-1) < M+\nred\]
or $M \geq R$.
Since $\delta_{\alpha^s}(\square+1,
  C(i)-1)=\delta_{\alpha^{s+1}}(\square+1,
  C(i)-1)$ we have that $T_{k,s}(\square') = T_{k,s+1}(\square')$ for
  all $\square' \in [\square+1,C(i)-1]$. Therefore at this point we are applying
  \Cref{alg:content} with $D = C(i), D' = \square$. We know that
  \[ T_{k,s}(C(i)) = i =
    T_\al(C(i))+\delta_{\alpha^s}\big(\partner(C(i)),C(i)\big) =
    T_\al(C(i))+\delta_{\alpha^s}(\square,C(i))\]
  where the last equality comes from \Cref{lem:crucial} \eqref{oaoa}.
  This means that $M+\nred \geq i$ and 
  \[ T_\al(C(i))+\delta_{\alpha^{s+1}}\big(\square,C(i)\big) =
    T_\al(C(i))+\delta_{\alpha^s}(\square,C(i))-1 = i-1 < M+\nred.\]
Thus $T_{k,s+1}(C(i)) = \ov{T_{k,s}(\square)} = i-1$ and at this step
we update $D = C(i)+1, D' = \square-1$, therefore $T_{k,s+1}(\square')
= T_{k,s}(\square')$ for all $1 \leq \square' < \square$ and $C(i) <
\square' \leq |\al|$. Comparing the resulting
$T_{k,s+1}$ with \ref{2.3} of \Cref{prop:Insertion} we conclude the
proof in this case. 

\vspace{10pt}

\textbf{\ref{3.1}}
We know that $e = \ov i$ for some $i \in
\CCC_{>0}$. \Cref{lem:crucial} \eqref{oaoa2} implies that
$\square +1 = C(x)$.
Indeed, $y < \ov{i} \leq x$, thus either $\square +1 = C(x)$ or
$\square +1 = \square'$, where $\square'$ is a box lying directly
under $C(x)$ and necessarily $x=\ov i$. Suppose that $\square +1 = \square'$. If $s=1$ then either there
exists $\square'' \in \I_i$ or $\mu_{\ov{i-\nred}}=\max_j
\alpha_j$. The first case contradicts \Cref{lem:crucial} \eqref{oaoa}
since
\[ \delta_{\al^s}(C(x),\square'') >
  \delta_{\al^s}(\square,\square'')\]
and the second case contradicts \Cref{lem:reduction}. Suppose that
$s>1$ and that $\square +1 = \square'$. Notice that \Cref{lem:crucial} \eqref{oaoa2} implies
that $T_{k,s}(\square'') = \ov i$ for all
$\square'' \in [C(x),\square'-1]$. If there exists $\square'' \in
\I_i$ then again
\[ \delta_{\al^s}(C(x),\square'') >
  \delta_{\al^s}(\square,\square'')\]
which contradicts \Cref{lem:crucial} \eqref{oaoa}. If $\I_i =
\emptyset$ then by the inductive hypothesis $T_{k,s}$ was obtained
as $\locins(T_{k,s-1})$, which corresponds to \ref{3.1} of
\Cref{prop:Insertion}. In this case $T_{k,s-1} =
\localshift^{-1}T_{k,s}$. Repeating this argument $s-1$ times we get
that
\[ T_{k,1} = \localshift^{1-s}T_{k,s}\]
thus $T_{k,1}$ is not authorized, which is a contradiction with
\Cref{lem:reduction}. This finishes the proof of our claim that
$\square +1 = C(x)$. 

We are going to show that
\begin{equation}
  \label{eq:!}
  T_{k,s}(\square'') = T_{k,s+1}(\square'')
  \end{equation}
for every $\square'' \in [1,|\al|]$. Comparing this with \ref{3.1} of
\Cref{prop:Insertion} we will conclude the
proof in this case. First, note that $x$ is barred. Otherwise
\[\square < \partner(C(x)) = \partner(\square+1)< \square+1\]
by \Cref{lem:crucial} \eqref{oaoa2}, and
this is clearly impossible. Note that
\[ \delta_{\al^{s}}(\square+1,\square'') =
  \delta_{\al^{s+1}}(\square+1,\square'')\]
for all $\square'' \in \I_{>0}$ and
\[ \delta_{\al^{s}}(\square,\square'') =
  \delta_{\al^{s+1}}(\square,\square'')\]
for all $\square'' \in \I_{>0}\setminus C$. Up to the step in\Cref{alg:content} when $D'\leq \square$ the construction of
$T_{k,s}$ and $T_{k,s+1}$ coincides. Since $m<i<n$ it is clear that
the transition from $D'>\square$ into $D'\leq\square$ necessarily
happens for $m<D<n$. In particular $D \in \I_{>0}\setminus C$ and
\[ \delta_{\al^{s}}(\square,\square'') =
  \delta_{\al^{s+1}}(\square,\square'').\]
In particular the construction of
$T_{k,s}$ and $T_{k,s+1}$ coincides at this step of \Cref{alg:content}, and trivially
coincides after achieving this step. This finishes the proof.

\vspace{10pt}
\textbf{\ref{3.2}}
We know that $e = \ov i$ for some $i \in
\CCC_{>0}$. \Cref{lem:crucial} \eqref{oaoa2} implies that
$\square +1 = C(n)$ since $\ov m < \ov i < n$. Therefore $T_{k,s}(\square+1) =
y$ and \Cref{maciekscor} implies that $C(n) = \max \I_n$ and $C(\ov n) = \min
\I_{\ov n}$. This is a consequence of the fact that
\[ \delta_{\al^s}\big(C (\overline{n} ),b'\big) > \delta_{\al^s}\big(C
  (\overline{n} ),C(n)\big)\]
for every $b' > C(n)$ and similarly
\[\delta_{\al^s}\big(b',C(n)\big) >\delta_{\al^s}\big(C (\overline{n}
  ),C(n)\big)\]
for every $b' < C(\ov n)$. Moreover, $C(n) = \partner(C(\ov n))$ by
\Cref{lem:crucial} \eqref{oaoa'}. We also note that for every $0< j
<b$ one has
$\delta_{\al^s}\big(C(\ov{n+j}),C(n)+1\big)-\delta_{\al^s}\big(C(\ov{n}),C(n)\big)
> j$
thus $T_{k,s}(C(n)+1) \geq n+b$ by \Cref{maciekscor}. Finally, since
$\square \in \I_{<0}$ and $\square+1 \in \I_{>0}$ we have that
\[ \delta_{\alpha^{s+1}}(\square', \square+1)
= \delta_{\alpha^{s+1}}(\square',\square+1)\]
for  $\square' \in [1,\square]\setminus C$ and
\[\delta_{\alpha^{s+1}}(\square', \square+1)
= \delta_{\alpha^{s+1}}(\square',\square+1)+1\]
for $\square' \in [1,\square]\cap C$. In particular $\nred = i-1$ thus
$\big[C(\ov{n+b-1}) -\mu_{\ov{n+b+1-i}},C(\ov{n})\big)$ are singles filled
by $\{\ov{n+1}^{\mu_{\ov{n+2-i}}},\dots,
\ov{n+b}^{\mu_{\ov{n+b+1-i}}}\}$ and $\big(C(\ov{n}),\square\big]$ are singles filled
by $\{\ov{i}^{\mu_1},\dots,
\ov{n}^{\mu_{\ov{n+1-i}}}\}$.
Therefore performing \Cref{alg:content} to obtain $T_{k,s+1}$
gives us the same result as in $T_{k,s}$ until $D = C(n)=\square+1,D'
= C(\ov{n})$. At this
moment $ M+\nred = n+1$, so since 
\[\delta_{\alpha^{s+1}}(D', \square+1)
= \delta_{\alpha^{s}}(D',\square+1)+1\] we have that $X = n+1 \nless
M+\nred$. Thus the interval
\[\big(C(\ov{n+b-1})-\mu_{\ov{n+b+1-i}},C(\ov{n})\big]\] in $T_{k,s+1}$
consists of single boxes filled by $\{\ov{n+1}^{\mu_{\ov{n+2-i}}},\dots, \ov{n+b}^{\mu_{\ov{n+b+1-i}}}\}$. After performing these steps we
have that \[D'=C(\ov{n+b-1})-\mu_{\ov{n+b+1-i}},M+\nred=n+b+1.\] Since $D'
< C(\ov{r})$ we have that 
\[X =\delta_{\alpha^{s+1}}(D', D)+T_{\al}(D) = n+b<M+\nred\] and
\[T_{k,s+1}(\square+1) = n+b, T_{k,s+1}(C(\ov{n+b-1}))= \ov{n+b}.\] At this
step of the algorithm  $D =
\square+2,D' = C(\ov{n+b-1})-\mu_{\ov{n+b+1-i}}-1$ and $M+\nred=n+b+1$, therefore we have the same parameters
of \Cref{alg:content} as at a certain point of \Cref{alg:content}
performed to construct $T_{k,s}$. Thus, all the other contents of
$T_{k,s+1}$ are the same as in $T_{k,s}$. Comparing resulting
$T_{k,s+1}$ with \ref{3.2} of \Cref{prop:Insertion} we conclude the
proof in this case.

\section*{Acknowledgments}
We thank C\'edric Lecouvey for many interesting
conversations and the anonymous reviewer for suggesting possible
extensions of our work. The SageMath computer algebra system
\cite{sagemath} has been used for experimentation leading up to many of the results
presented in the paper.

\bibliographystyle{amsalpha}


\def\cprime{$'$}
\providecommand{\bysame}{\leavevmode\hbox to3em{\hrulefill}\thinspace}
\providecommand{\MR}{\relax\ifhmode\unskip\space\fi MR }
\providecommand{\MRhref}[2]{%
  \href{http://www.ams.org/mathscinet-getitem?mr=#1}{#2}
}
\providecommand{\href}[2]{#2}

\end{document}